\documentclass[a4paper]{amsart}
\usepackage{amsmath}
\usepackage{booktabs}
\usepackage{amssymb}
\usepackage[small,nohug,heads=LaTeX]{diagrams}
\usepackage{graphicx}
\usepackage{pstricks}
\usepackage{verbatim}
\usepackage{enumerate}

\newtheorem{thm}{Theorem}[section]
\newtheorem{cor}[thm]{Corollary}
\newtheorem{lem}[thm]{Lemma}
\newtheorem{prop}[thm]{Proposition}
\theoremstyle{definition}
\newtheorem{defn}[thm]{Definition}
\theoremstyle{remark}

\newtheorem{example}[thm]{Example}
\numberwithin{equation}{section}

\newcommand{\bC}{\mathbb{C}}

\newcommand{\bF}{\mathbb{F}}

\newcommand{\bN}{\mathbb{N}}

\newcommand{\bQ}{\mathbb{Q}}
\newcommand{\bR}{\mathbb{R}}

\newcommand{\bZ}{\mathbb{Z}}

\newcommand{\gC}{\bold{C}}

\newcommand{\gK}{\bold{K}}

\newcommand{\gM}{\bold{M}}

\newcommand{\gS}{\bold{S}}

\newcommand{\gX}{\bold{X}}
\newcommand{\gY}{\bold{Y}}

\newcommand{\MT}[2]{\bold{MT #1}(#2)}

\newcommand\lra{\longrightarrow}

\newcommand\Diff{\mathrm{Diff}}
\newcommand\Emb{\mathrm{Emb}}
\newcommand\Bun{\mathrm{Bun}}

\newcommand{\hcoker}{/\!\!/}

\newarrow {Onto} ----{>>}
\newarrow {Equals} ===={=}

\newcommand{\CircNum}[1]{\ooalign{\hfil\raise .00ex\hbox{\scriptsize #1}\hfil\crcr\mathhexbox20D}}

\newcommand{\X}{\mathbf{X}}

\newcommand{\Spin}{\mathrm{Spin}}

\newcommand{\divides}{\mid}
\newcommand{\ndivides}{\nmid}

\newcommand{\Pic}{\mathrm{Pic}}
\newcommand{\Td}{\mathrm{Td}}
\renewcommand{\top}{\mathrm{top}}
\newcommand{\hol}{\mathrm{hol}}
\newcommand{\alg}{\mathrm{alg}}

\mathchardef\ordinarycolon\mathcode`\:
\mathcode`\:=\string"8000
\begingroup \catcode`\:=\active
  \gdef:{\mathrel{\mathop\ordinarycolon}}
\endgroup

\title[Moduli of $r$-Spin Riemann surfaces]{The Picard group of the moduli space of $r$-Spin Riemann surfaces}

\author{Oscar Randal-Williams}
\thanks{The author was supported by the EPSRC PhD Plus scheme, ERC Advanced Grant No.\ 228082, and the Danish National Research Foundation through the Centre for Symmetry and Deformation.}
\email{o.randal-williams@math.ku.dk}
\address{Department of Mathematical Sciences\\
University of Copenhagen\\
Universitetsparken 5\\
DK-2100 K{\o}benhavn {\O}\\
Denmark}
\date{\today}

\begin{document}

\begin{abstract}
We have recently proved a homological stability theorem for moduli spaces of $r$-Spin Riemann surfaces, which in particular implies a Madsen--Weiss theorem for these moduli spaces. This allows us to effectively study their stable cohomology, and to compute their stable rational cohomology and their integral Picard groups. Using these methods we give a complete description of their integral Picard groups for genus at least 9 in terms of geometrically defined generators, and determine the relations between them.
\end{abstract}
\maketitle

\section{Introduction}

Let $\gM_g$ be the moduli space of Riemann surfaces of genus $g$. This is the complex orbifold obtained from Teichm\"{u}ller space $\mathcal{T}_g$ as the quotient by the action of the mapping class group $\Gamma_g := \pi_0(\Diff^+(\Sigma_g))$, the group of isotopy classes of diffeomorphisms of a standard smooth closed surface of genus $g$. As Teichm\"{u}ller space is contractible, the orbifold fundamental group of $\gM_g$ is $\Gamma_g$. We wish to make a similar definition of the moduli space of $r$-Spin Riemann surfaces of genus $g$, but we must be careful as there are competing definitions. We will give two, and discuss which is more correct in \S \ref{sec:ModularInterpretation}. Throughout we consider only surfaces of genus at least two.

Fix a smooth almost complex surface $\Sigma$, and let $\Spin^r(\Sigma)$ denote the \emph{groupoid of $r$-Spin structures on $\Sigma$}. More precisely, its objects are pairs $\zeta = (L \to \Sigma, \varphi : T\Sigma \cong  L^{\otimes r})$ of a complex line bundle $L$ on $\Sigma$ and an isomorphism of complex line bundles from the tangent bundle of $\Sigma$ to the $r$-th power of $L$. The morphisms of this groupoid from $(L, \varphi)$ to $(L', \varphi')$ are given by the isomorphisms $\psi : L \cong L'$ such that $\psi^{\otimes r} \circ \varphi = \varphi'$. Given a diffeomorphism $f : \Sigma \to \Sigma'$ that preserves the almost complex structure, we define
$$f^*(L', \varphi') = (f^*L', T\Sigma \overset{Df}\to f^*T\Sigma' \overset{f^*\varphi'}\to f^*(L^{\otimes r}) \cong (f^*L)^{\otimes r}).$$
This provides an action of the group of diffeomorphisms of $\Sigma$ on the set of $r$-Spin structures. The following lemma was proved in \cite[Theorem 2.9]{RWFramedPinMCG}, and describes the components of the groupoid $\Spin^r(\Sigma)$ modulo the action of the mapping class group.
\begin{lem}\label{lem:Components}
Suppose that $g \geq 2$. The set $\pi_0\Spin^r(\Sigma_g)$ of isomorphism classes of $r$-Spin structures on $\Sigma_g$ is non-empty if and only if $r$ divides $\chi(\Sigma_g) = 2-2g$, in which case it consists of $r^{2g}$ elements. If $r$ is odd, $\pi_0\Spin^r(\Sigma_g) / \Gamma_g$ consists of a single element. If $r$ is even, $\pi_0\Spin^r(\Sigma_g) / \Gamma_g$ consists of two elements, and the map
$$\pi_0\Spin^r(\Sigma_g) / \Gamma_g \to \pi_0\Spin^2(\Sigma_g) / \Gamma_g$$
is a bijection. The two elements of $\pi_0\Spin^2(\Sigma_g) / \Gamma_g$ are distinguished by the Arf invariant of a quadratic form describing the Spin structure.
\end{lem}
The remaining structure of the groupoid is described by the following lemma.
\begin{lem}
Every element of $\Spin^r(\Sigma)$ has automorphism group $\bZ/r$.
\end{lem}

We may now give our first definition. We can choose an $r$-Spin structure $\zeta$ on $\Sigma_g$, and let $G_g^{1/r}(\zeta)$ be the subgroup of $\Gamma_g$ that preserves $\zeta$ up to isomorphism, that is, those diffeomorphisms $f$ such that $f^*(\zeta) \cong \zeta$. Equivalently it is the stabiliser of $\zeta$ for the action of $\Gamma_g$ on $\pi_0\Spin^r(\Sigma_g)$. We define
$$\widetilde{\gM}_g^{1/r}(\zeta) := \mathcal{T}_g \hcoker G_g^{1/r}(\zeta),$$
where the quotient is taken in the orbifold sense. As the set $\pi_0\Spin^r(\Sigma_g)$ is finite, $G_g^{1/r}(\zeta)$ is a finite index subgroup of the mapping class group, and hence
$$\widetilde{\gM}_g^{1/r}(\zeta) \lra \gM_g$$
is a finite covering map. Readers who are unenthusiastic about orbifolds may instead take the homotopy quotient, that is, the Borel construction $EG_g^{1/r}(\zeta) \times_{G_g^{1/r}(\zeta)} \mathcal{T}_g$, or the stack quotient.

For our second definition, we employ another group $\Gamma^{1/r}_{g}(\zeta)$ which is slightly more difficult to describe: to do so, we will first define a topological group $\Diff(\Sigma_g, \zeta)$. As a set it consists of pairs $(f \in \Diff^+(\Sigma_g), \rho : f^*\zeta \cong \zeta)$, and is equipped with the group law defined by the formula
$$(f, \rho) \cdot (g, \sigma) = (f \circ g, (f \circ g)^*\zeta \cong f^*g^* \zeta \overset{f^*\sigma}\cong f^*\zeta \overset{\rho}\cong \zeta).$$
The evident homomorphism $\Diff(\Sigma_g, \zeta) \to \Diff^+(\Sigma_g)$ has image those path components contained in $G_g^{1/r}(\zeta)$, and has kernel $\bZ/r$. We topologise $\Diff(\Sigma_g, \zeta)$ as a covering group of its image, and define
$$\Gamma^{1/r}_{g}(\zeta) := \pi_0(\Diff(\Sigma_g, \zeta)),$$
which is isomorphic to the quotient of $\Diff(\Sigma_g, \zeta)$ by the normal subgroup $\Diff_0(\Sigma_g, \zeta)$ given by the path component of the identity. By a theorem of Earle and Eells \cite{EE}, the group $\Diff^+_0(\Sigma_g)$ is contractible for $g > 1$, and so has no non-trivial connected covering groups. Thus $\Diff_0(\Sigma_g, \zeta) \cong \Diff^+_0(\Sigma_g)$, and hence the action of $\Diff(\Sigma_g, \zeta)$ on the space $M(\Sigma_g)$ of complex structures on $\Sigma_g$ descends to an action of $\Gamma_g^{1/r}(\zeta)$ on Techm{\"u}ller space $\mathcal{T}_g = M(\Sigma_g) / \Diff_0(\Sigma_g, \zeta)$, and we may define the orbifold
$$\gM_g^{1/r}(\zeta) := \mathcal{T}_g \hcoker \Gamma^{1/r}_g(\zeta).$$
By definition, the group $\Gamma^{1/r}_{g}(\zeta)$ fits into the central extension
\begin{equation}\label{eq:GroupExt}
0 \lra \bZ/r \lra \Gamma^{1/r}_{g}(\zeta) \lra G_g^{1/r}(\zeta) \lra 0
\end{equation}
and so acts through $G_g^{1/r}(\zeta)$, and hence $\Gamma_g$, on Teichm\"{u}ller space. In particular both groups act by biholomorphisms (as $\Gamma_g$ does), and so $\widetilde{\gM}_g^{1/r}(\zeta)$ and $\gM_g^{1/r}(\zeta)$ are complex orbifolds. The natural map
\begin{equation}\label{eq:Gerbe}
\gM_g^{1/r}(\zeta) \lra \widetilde{\gM}_g^{1/r}(\zeta)
\end{equation}
has the structure of a $\bZ/r$-gerbe, which follows from the extension (\ref{eq:GroupExt}). Furthermore, as $\mathcal{T}_g$ is contractible these are both orbifold Eilenberg--MacLane spaces, and so all statements about their homology apply equally well to the group homology of the groups $G_g^{1/r}(\zeta)$ and $\Gamma_g^{1/r}(\zeta)$. 

If the $r$-Spin structures $\zeta$ and $\zeta'$ are identified under the action of $\Gamma_g$, then the groups $G_g^{1/r}(\zeta)$ and $G_g^{1/r}(\zeta')$ are conjugate in $\Gamma_g$ and so the corresponding $r$-Spin moduli spaces---of both types---are homeomorphic. By Lemma \ref{lem:Components}, when $r$ is odd there is, up to homeomorphism, a single moduli space of each type, which we write as $\widetilde{\gM}_g^{1/r}$ and $\gM_g^{1/r}$. When $r$ is even there are two of each type, which we write as $\widetilde{\gM}_g^{1/r}[\epsilon]$ and $\gM_g^{1/r}[\epsilon]$ for $\epsilon=0$ or $1$. In this case we write $\widetilde{\gM}_g^{1/r}$ and $\gM_g^{1/r}$ for the disjoint union of the two path components.

\subsection{Modular interpretation of $\widetilde{\gM}_g^{1/r}$ and $\gM_g^{1/r}$}\label{sec:ModularInterpretation}

The (orbi)bundle $\pi:\gC_g^{1/r} \to \gM_g^{1/r}$ is equipped with a holomorphic line bundle $L \to \gC_g^{1/r}$ and an isomorphism $L^{\otimes r} \cong \omega_\pi$, and so is a family of $r$-Spin Riemann surfaces over $\gM_g^{1/r}$. Furthermore, it is universal with this property. This may be seen as follows: an $r$-Spin structure $\zeta$ gives a line bundle $L$ on the family of Riemann surfaces $\pi:\mathcal{C}_g \to \mathcal{T}_g$ and an isomorphism $\varphi: T_v \cong L^{\otimes r}$. The line bundle $T_v \to \mathcal{C}_g$ is fibrewise holomorphic and so $L \to \mathcal{C}_g$ is too. The element $(f, \rho) \in \Gamma_g^{1/r}(\zeta)$ acts on $\mathcal{T}_g$ and $\mathcal{C}_g$, and gives an isomorphism $\rho : f^*L \cong L$ such that $\rho^{\otimes r}$ commutes with $\varphi$, and hence it descends to the orbifold quotient.

The orbifold $\widetilde{\gM}_g^{1/r}$ has a different modular interpretation. The (orbi)bundle $\pi:\gC_g \to \gM_g$ has a fibrewise Picard variety $\mathrm{Pic}(\gC_g) \to \gM_g$, which has a canonical section given by the canonical divisor $\omega_\pi$. The universal property of the family $\widetilde{\gC}_g^{1/r} \to \widetilde{\gM}_g^{1/r}$ is that in addition it has another section $\ell$ such that $r\ell = \omega_\pi$. Thus each fibre admits the structure of an $r$-Spin Riemann surface, but there is no line bundle $L \to \gC_g^{1/r}$ whose $r$-th power is isomorphic to the (co)tangent bundle on each fibre. The obstruction to a section of the fibrewise Picard variety of a family $\pi:E \to B$ of Riemann surfaces extending to a line bundle on the total space has been studied by Ebert and the author \cite{ERW10}: there is a unique obstruction in $H^3(B;\bZ)$. In the case of $\widetilde{\gM}_g^{1/r}$ it is the class given by $\beta(c)$, the Bockstein to integral cohomology of the class $c \in H^2(\widetilde{\gM}_g^{1/r};\bZ/r)$ classifying the extension (\ref{eq:GroupExt}), or equivalently classifying the gerbe (\ref{eq:Gerbe}).

Considering these two different modular interpretations, we take the view that $\gM_g^{1/r}$ is the correct notion of a moduli space of $r$-Spin Riemann surfaces, and $\widetilde{\gM}_g^{1/r}$ is not; we call the latter the \emph{moduli space of $r$-theta-characteristics}, as $\widetilde{\gM}_g^{1/2}$ classifies families of Riemann surfaces with a theta-characteristic on each fibre.

\subsection{``Mumford conjecture" for moduli spaces of $r$-Spin Riemann surfaces}

From now on we denote by $\gM_g^{1/r}[\epsilon]$ either the whole of $\gM_g^{1/r}$ if $r$ is odd, or the component of Arf invariant $\epsilon$ if $r$ is even. By writing $\gM_g^{1/r}[\epsilon]$ we always imply that the space is non-empty, i.e.\ that $r \divides \chi(\Sigma_g)$. Whenever we write the cohomology of an orbifold, we always mean cohomology in the orbifold sense, not that of the coarse moduli space; of course if we take rational coefficients these coincide.

The Mumford--Morita--Miller classes $\kappa_i \in H^{2i}(\gM_g;\bZ)$ may be defined as the pushforward
$$\kappa_i := \pi_!(c_1(T_v)^{i+1})$$
where $\pi : \gC_{g} \to \gM_g$ is the universal family over $\gM_g$, and $T_v$ is the vertical tangent bundle (which is a complex line bundle). We will adopt the convention of always denoting families of Riemann surfaces by $\pi$ and their vertical tangent bundles by $T_v$, without further note. The \emph{Mumford conjecture}, proved by Madsen and Weiss \cite{MW}, states that the homomorphism
$$\bQ[\kappa_1, \kappa_2, \kappa_3, ...] \lra H^*(\gM_g;\bQ)$$
is an isomorphism in a certain range of degrees, which increases with $g$. Our first result is the analogue of this theorem for moduli spaces of $r$-Spin Riemann surfaces: the classes $\kappa_i$ may be pulled back to $\gM_g^{1/r}$ and we have

\begin{thm}\label{thm:MumfordConjecture}
The map
$$\bQ[\kappa_1, \kappa_2, \kappa_3, ...] \lra H^*(\gM_g^{1/r}[\epsilon];\bQ)$$
is an isomorphism in degrees $5* \leq 2g-7$. In fact, the map $\gM_g^{1/r}[\epsilon] \to \gM_g$ induces a homology isomorphism with $\bZ[1/r]$-coefficients in these ranges of degrees. 
Both statements are also true for $\widetilde{\gM}_g^{1/r}[\epsilon]$, as this is $\bZ[1/r]$-homology equivalent to $\gM_g^{1/r}[\epsilon]$ by the extension (\ref{eq:GroupExt}).
\end{thm}


\subsection{Low-dimensional homology}

The main result of this paper is a computation of the Picard group of $\gM_g^{1/r}$ and a determination of a presentation for it. As these results are of interest in algebraic geometry, but the method of proof is via homotopy theory, we have endeavoured to give as detailed results as we can in the introduction to hopefully allow geometers to make use of them. This accounts for the length of this introduction.

We first turn our attention to the low-dimensional integral (co)homology of $\gM_g^{1/r}$, which is the same as the integral (co)homology of the groups $\Gamma_g^{1/r}(\zeta)$.

\begin{thm}[Low-dimensional homology]\label{thm:LowDimHomology}
Let $g \geq 6$. Then the first integral homology of $\gM_g^{1/r}[\epsilon]$ is given by
$$H_1(\gM_g^{1/r}[\epsilon];\bZ) \cong \begin{cases}
\bZ/4 & r \equiv 2 \,\, \mathrm{mod} \,\, 4\\
\bZ/8 & r \equiv 0 \,\, \mathrm{mod} \,\, 4\\
0 & \text{else}
\end{cases}
\oplus
\begin{cases}
\bZ/3 & r \equiv 0 \,\, \mathrm{mod} \,\, 3\\
0 & \text{else}
\end{cases}.$$

Let $g \geq 9$. Then the second rational homology of $\gM_g^{1/r}[\epsilon]$ has rank one. Thus,
$$H^2(\gM_g^{1/r}[\epsilon];\bZ) \cong \bZ \oplus
\begin{cases}
\bZ/4 & r \equiv 2 \,\, \mathrm{mod} \,\, 4\\
\bZ/8 & r \equiv 0 \,\, \mathrm{mod} \,\, 4\\
0 & \text{else}
\end{cases}
\oplus
\begin{cases}
\bZ/3 & r \equiv 0 \,\, \mathrm{mod} \,\, 3\\
0 & \text{else}
\end{cases}.$$
\end{thm}

\subsection{Identification of classes in $H^2(\gM_g^{1/r}[\epsilon];\bZ)$}
The description of the second cohomology of $\gM_g^{1/r}[\epsilon]$ in Theorem \ref{thm:LowDimHomology}, while interesting in itself, is of limited use if we do not have firm control on individual elements in this group, and so we now turn to describing elements of this group.

The group $H^2(\gM_g;\bZ)$ is free of rank one (for $g \geq 3$), and contains the two natural classes
$$\kappa_1 := \pi_!(c_1(T_v)^2) \quad \quad\quad \lambda := c_1(\pi^K_!(T_v^*))$$
where the symbol $\pi_!^K$ in the definition of the Hodge class $\lambda$ is the complex $K$-theory pushforward. These are related by the equation $\kappa_1 = 12 \lambda$, and $\lambda$ generates $H^2(\gM_g;\bZ)$.

The universal family $\pi : \gC_g^{1/r} \to \gM_g^{1/r}$ has a vertical tangent bundle $T_v$, and also another complex line bundle $L$, equipped with an isomorphism $L^{\otimes r} \cong T_v$. Thus we may define classes
\begin{equation}\label{eq:TautClasses}
\kappa_1^{a/r} := \pi_!(c_1(L^{\otimes a})^2) \quad \quad\quad \lambda^{-a/r} := c_1(\pi_!^K(L^{\otimes a})).
\end{equation}
These satisfy $\kappa_1^{r/r} = \kappa_1$ and $\lambda^{r/r} = \lambda$. Furthermore, if $r \divides r'$ then the natural map $[r'/r] : \gM_g^{1/r'} \to \gM_g^{1/r}$ pulls back $\kappa_1^{a/r}$ to $\kappa_1^{(a(r'/r))/r'}$, and similarly for the $\lambda^{-a/r}$, so there is no ambiguity in writing the superscripts as rational numbers.

When $r$ is even there is a further class we may define. In this case, an $r$-Spin structure has an underlying 2-Spin structure, and so the map $\gC_g^{1/r} \to \gM_g^{1/r}$ is oriented in real $K$-theory. Thus we may define a class
$$\xi := \pi_!^{KO}(1) \in KO^{-2}(\gM_g^{1/r}),$$
that is, a map $\xi: \gM_g^{1/r} \to O/U$. This is represented by a (virtual dimension zero) complex vector bundle with a trivialisation of the underlying real bundle, and such bundles have canonical choices of half the first Chern class so we may define the element $\tfrac{c_1}{2}(\xi)$ in the integral cohomology of $\gM_g^{1/r}$. The Spin structure on $T_v$ used to form the pushforward is $L^{\otimes \tfrac{r}{2}}$, so the underlying complex vector bundle of $\xi$ is $\pi_!^K(L^{\otimes \frac{r}{2}})$ and $c_1(\xi) = \lambda^{-1/2}$. Using this we define
$$\mu := \tfrac{c_1}{2}(\xi) + 6\lambda^{1/2},$$
which has the property $2\mu = \lambda^{-1/2} + 12\lambda^{1/2}$. The reason for taking this class and not simply $\tfrac{c_1}{2}(\xi)$ is that certain formul{\ae} occurring later on will be clearer.

Our first result concerns the divisibility of these classes in the torsion-free quotient of $H^2(\gM_g^{1/r}[\epsilon];\bZ)$, which by Theorem \ref{thm:LowDimHomology} is a free abelian group of rank one. In the following, we fix a generator $g$ of the torsion-free quotient so that the Hodge class $\lambda$ is a \emph{positive} multiple of $g$. We say a class $x$ is divisible by precisely $D$ when $x = Dg$.

\begin{thm}\label{thm:TorsionFreeDivisibility}
The group $H^2(\gM_g;\bZ)$ injects into $H^2(\gM_g^{1/r}[\epsilon];\bZ)$. In the torsion-free quotient of $H^2(\gM_g^{1/r}[\epsilon];\bZ)$, the Hodge class $\lambda$ is divisible by precisely $\tfrac{U_rr^2}{12}$, where
$$U_r = \begin{cases}
2 & 12 \divides r\\
4 & 4 \ndivides r, 3 \divides r\\
6 & 4 \divides r, 3 \ndivides r\\
12 & 4 \ndivides r, 3 \ndivides r.
\end{cases}$$
More generally,  $\kappa_1^{a/r}$ is divisible by precisely $a^2 U_r$, and $\lambda^{a/r}$ is divisible by precisely $\frac{U_r}{12}(r^2 - 6ar+6a^2)$. When $\mu$ is defined it is divisible by precisely $-\frac{U_r r^2}{48}$.
\end{thm}

It is tedious but not difficult to see that for each fixed $r$ all these divisibilities have no common factor, and hence the classes we have defined generate the torsion-free quotient.

\begin{cor}\label{cor:TorsionFreeGeneration}
The elements $\{\lambda^{a/r}, \kappa_1^{a/r}\}$ (and $\mu$ if it is defined) generate the torsion-free quotient of $H^2(\gM_g^{1/r}[\epsilon];\bZ)$.
\end{cor}

Having understood the divisibilities of the classes in the torsion-free quotient, it is easy to produce torsion classes as linear combinations of the $\kappa_1^{a/r}$, $\lambda^{a/r}$ and $\mu$ which vanish in the torsion-free quotient. If we define $U_{a/r, b/r} := \gcd(r^2-6ar+6a^2, r^2-6br+6b^2)$ then
$$t^{a/r, b/r} := \frac{r^2-6br+6b^2}{U_{a/r, b/r}} \lambda^{a/r} - \frac{r^2-6ar+6a^2}{U_{a/r, b/r}}\lambda^{b/r}$$
and
$$t^{a/r} := \frac{12}{\gcd(12, r^2-6ar+6a^2)}\lambda^{a/r} - \frac{r^2-6ar+6a^2}{\gcd(12, r^2-6ar+6a^2)} \kappa_1^{1/r}$$
are integral cohomology classes which are trivial in the torsion-free quotient, and hence are torsion. When $r$ is even, there is also the torsion class
$$t := \frac{48}{\gcd(r^2, 48)}\mu + \frac{r^2}{\gcd(r^2,48)}\kappa_1^{1/r}.$$
In order to determine when torsion classes are non-trivial, we prove the following detection theorem.

\begin{thm}\label{thm:TorsionDetection}
For $g \geq 9$ there is a canonical homomorphism
$$\varphi: H^2(\gM_g^{1/r}[\epsilon];\bZ) \lra \bZ/24$$
which is injective when restricted to the torsion subgroup. It sends $\lambda^{a/r}$ to $2$, $\kappa_1^{a/r}$ to $0$, and if $\mu$ is defined it sends it to $1$.
\end{thm}

Along with an analysis of the coefficients in the definition of $t^{a/r}$ and $t$, this theorem implies the following.

\begin{cor}\label{cor:TorsionGenerators}
If $r$ is odd, any $t^{a/r}$ generates the torsion subgroup. If $r \equiv 2 \,\,\mathrm{mod}\,\, 4$ then $t^{0/r}$ generates the torsion subgroup. If $r \equiv 0 \,\,\mathrm{mod}\,\, 4$ then $t$ generates the torsion subgroup.
\end{cor}

Combining Corollaries \ref{cor:TorsionFreeGeneration} and \ref{cor:TorsionGenerators}, we can make the following observation. 

\begin{cor}
The elements $\{\lambda^{a/r}, \kappa_1^{a/r}\}$ (and $\mu$ if it is defined) generate the group $H^2(\gM_g^{1/r}[\epsilon];\bZ)$. All relations between them are implied by Theorem \ref{thm:TorsionFreeDivisibility} and Theorem \ref{thm:TorsionDetection}.
\end{cor}

\begin{example}[2-Spin]\label{ex:rEq2}
Consider the classes $\lambda$, $\lambda^{1/2}$, $\mu$, $\kappa_1$ and $\kappa_1^{1/2}$. The number $U_2$ is 12, so in the torsion-free quotient $\lambda$ is divisible by 4, $\lambda^{1/2}$ is divisible by $-2$, $\mu$ is divisible by $-1$, $\kappa_1$ is divisible by $48$, and $\kappa_1^{1/2}$ is divisible by $12$. The class $t^{1/2, 0/2}$ is $2\lambda^{1/2} + \lambda$ and maps to $6 \in \bZ/24$ an order 4 element, so it generates the torsion subgroup. Thus a presentation of the second integral cohomology of $\gM_g^{1/2}[\epsilon]$ in the stable range is
$$\langle \lambda, \mu \,\, \vert \,\, 4(\lambda + 4\mu) \rangle.$$
\end{example}

\begin{example}[3-Spin]\label{ex:rEq3}
Consider the classes $\lambda$, $\lambda^{1/3}$, $\lambda^{2/3}$,  $\kappa_1$ and $\kappa_1^{1/3}$. The number $U_3$ is 4, so in the torsion-free quotient $\lambda$ is divisible by 3, $\lambda^{1/3}$ is divisible by $-1$, $\lambda^{2/3}$ is divisible by $-1$, $\kappa_1$ is divisible by $36$, and $\kappa_1^{1/3}$ is divisible by $4$. The class $t^{1/3, 0/3}$ is $3\lambda^{1/3} + \lambda$ and maps to $8 \in \bZ/24$ an element of order 3, so it generates the torsion subgroup. Thus a presentation of the second integral cohomology of $\gM_g^{1/3}[\epsilon]$ in the stable range is
$$\langle \lambda, \lambda^{1/3} \,\, \vert \,\, 3(\lambda + 3\lambda^{1/3}) \rangle.$$
\end{example}

\begin{example}[4-Spin]\label{ex:rEq4}
Consider the classes $\lambda$, $\lambda^{1/4}$, $\lambda^{1/2}$, $\lambda^{3/4}$, $\mu$ and the $\kappa_1^{a/4}$. The number $U_4$ is 6, so in the torsion-free quotient $\lambda$ is divisible by $8$, $\lambda^{1/4}$ is divisible by $-1$, $\lambda^{1/2}$ is divisible by $-4$, $\lambda^{3/4}$ is divisible by $-1$, and $\mu$ is divisible by $-2$. The class $t^{1/4, 0/4}$ is $8\lambda^{1/4} + \lambda$ and maps to $18 \in \bZ/24$ an element of order 4. The class $t^{1/2, 0/4}$ is $2\lambda^{1/2} + \lambda$ and maps to $6 \in \bZ/24$ an element of order 4. The class $t^{1/4, 1/2}$ is $\lambda^{1/2}-4\lambda^{1/4}$ and maps to $-6 = 18 \in \bZ/24$ an element of order 4. The class $t^{1/4}$ is $6\lambda^{1/4}+\kappa_1^{1/4}$ and maps to $12 \in \bZ/24$ an element of order 2. The class $t^{1/2}$ is $3\lambda^{1/2} + 2\kappa_1^{1/4}$ and maps to $6 \in \bZ/24$ an element of order 4. Finally, the class $t$ is $3\mu + \kappa_1^{1/4}$ and maps to $3 \in \bZ/24$ an element of order 8, so generates the torsion subgroup. Similarly, the class $\mu - 2\lambda^{1/4}$ is torsion and maps to $-3 \in \bZ/24$ so generates the torsion subgroup. Thus a presentation of the second integral cohomology of $\gM_g^{1/4}[\epsilon]$ in the stable range is
$$\langle \mu, \lambda^{1/4} \,\, \vert \,\, 8(\mu -2\lambda^{1/4}) \rangle.$$
\end{example}

\subsection{Low-dimensional homology of the space of $r$-theta-characteristics}\label{sec:HomologyThetaCharacteristics}

The extension (\ref{eq:GroupExt}) and the known abelianisation of $\Gamma_g^{1/r}(\zeta)$ from Theorem \ref{thm:LowDimHomology} imply a calculation of the abelianisation of $G_g^{1/r}(\zeta)$ as long as one can understand the effect of the map $\bZ/r \to \Gamma_g^{1/r}(\zeta)$ on abelianisations. We explain in \S \ref{sec:ThetaChar} how to compute the effect of this map, but the formul{\ae} are complicated and the abelianisation of $G_g^{1/r}(\zeta)$ depends sensitively on $r$, $g$ and $\mathrm{Arf}(\zeta)$, so it is difficult to give a general statement. However for any particular $r$ it is not a difficult calculation.

\begin{example}
The group $H_1(\widetilde{\gM}_g^{1/2}[\epsilon];\bZ)$ is $\bZ/4$ for $g \geq 9$, and the group $H_1(\widetilde{\gM}_g^{1/3};\bZ)$ is $\bZ/3$ for $g \geq 9$. The group $H_1(\widetilde{\gM}_g^{1/4}[\epsilon];\bZ)$ is $\bZ/8$ if $\epsilon = 0$, and is $\bZ/4$ otherwise, assuming that $g \geq 9$.
\end{example}

From the abelianisation of $G_g^{1/r}(\zeta)$ we may deduce the second integral cohomology of $\widetilde{\gM}_g^{1/r}[\epsilon]$ as an abstract group, but the methods of \S \ref{sec:ThetaChar} in fact allow us to compute the effect of the (injective) map
$$H^2(\widetilde{\gM}_g^{1/r}[\epsilon];\bZ) \lra H^2({\gM}_g^{1/r}[\epsilon];\bZ).$$

\begin{example}
As a subgroup of $H^2({\gM}_g^{1/2}[\epsilon];\bZ) = \langle \lambda, \mu \,\, \vert \,\, 4(\lambda + 4\mu) \rangle$, the second cohomology of $\widetilde{\gM}_g^{1/2}[\epsilon]$ is the whole group if $\epsilon=0$ and $\langle \lambda, 2\mu \,\, \vert \,\, 4(\lambda + 4\mu) \rangle$ if $\epsilon=1$.
\end{example}

The first rational cohomology of $G_g^{1/r}(\zeta)$ is stably trivial, and the second rational cohomology stably has rank one, giving further support to the general belief that this is true of all finite-index subgroups of the mapping class group.

\subsection{Picard groups and Neron--Severi groups}

The \emph{topological Picard group} $\mathrm{Pic}_{\top}(X \hcoker G)$ of the orbifold $X \hcoker G$ is the set of isomorphism classes of $G$-equivariant complex line bundles on $X$, which forms an abelian group under tensor product of line bundles. The first Chern class provides a homomorphism
$$c_1 : \mathrm{Pic}_{\top}(X \hcoker G) \lra H^2(X \hcoker G;\bZ)$$
which is in fact an isomorphism \cite[Lemma 5.1]{ERW10}. 

The orbifolds $\widetilde{\gM}_g^{1/r}$ and $\gM_g^{1/r}$ have a complex structure, and so they also have \emph{holomorphic Picard groups}: these are the sets of isomorphism classes of $G_g^{1/r}$- or $\Gamma_g^{1/r}$-equivariant holomorphic line bundles on $\mathcal{T}_g$, under tensor product of holomorphic line bundles. For a holomorphic orbifold $\gX$ there is a map
$$\mathrm{Pic}_{\hol}(\gX) \to \mathrm{Pic}_{\top}(\gX)$$
with kernel $\mathrm{Pic}_{\hol}^0(\gX)$ consisting of the topologically trivial holomorphic line bundles. The \emph{analytic Neron--Severi group} $\mathcal{NS}(\gX)$ is defined to be the quotient group $\mathrm{Pic}_{\hol}(\gX) / \mathrm{Pic}_{\hol}^0(\gX)$, so there is a sequence of maps
$$\mathrm{Pic}_{\hol}(\gX) \lra \mathcal{NS}(\gX) \lra \mathrm{Pic}_{\top}(\gX) \overset{c_1}\lra H^2(\gX;\bZ)$$
where the first map is an epimorphism, the second is a monomorphism and by the discussion above the last is an isomorphism.



The orbifold $\widetilde{\gM}_g^{1/r}[\epsilon]$ has the structure of a \emph{quasi-projective orbifold}, as it has a finite cover by a quasi-projective variety. Thus it has an \emph{algebraic Picard group} $\Pic_{\alg}(\widetilde{\gM}_g^{1/r}[\epsilon])$ consisting of holomorphic line bundles which are algebraic on all finite quasi-projective covers. The orbifold $\gM_g^{1/r}[\epsilon]$ also has an algebraic Picard group, by virtue of being a $\bZ/r$-gerbe over the quasi-projective orbifold $\widetilde{\gM}_g^{1/r}[\epsilon]$; we define it in \S \ref{sec:AlgPicGp}.

\begin{thm}\label{thm:PicardMain}
Consider the maps
$$\mathrm{Pic}_{\alg}(\gM_g^{1/r}[\epsilon]) \lra \mathrm{Pic}_{\hol}(\gM_g^{1/r}[\epsilon]) \lra \mathcal{NS}(\gM_g^{1/r}[\epsilon]) \lra H^2(\gM_g^{1/r}[\epsilon];\bZ).$$
As long as $g \geq 9$, the composition is an isomorphism, the second map is surjective, and the last map is an isomorphism. The same holds for $\widetilde{\gM}_g^{1/r}[\epsilon]$.
\end{thm}

Hence Theorem \ref{thm:LowDimHomology} computes both the topological and algebraic Picard groups and the analytic Neron--Severi group of $\gM_g^{1/r}$, and the methods of \S \ref{sec:ThetaChar} do the same for $\widetilde{\gM}_g^{1/r}$.

\subsection{Relation to the work of Jarvis}

In the algebro-geometric setting, Jarvis \cite{JarvisGeom} has constructed $\overline{\mathfrak{S}}_g^{1/r}$, a smooth proper Deligne--Mumford stack over $\bZ[1/r]$ compactifying the stack $\mathfrak{S}_g^{1/r}$ which parametrises families of $r$-Spin curves (which is a curve $C$ with a line bundle $L$ and an isomorphism $\varphi: L^{\otimes r} \cong \omega$). For the purposes of this discussion, we assume that the quasi-projective orbifold $\gM_g^{1/r}$ represents the smooth Deligne--Mumford stack $\mathfrak{S}_g^{1/r} \otimes_{\bZ[1/r]} \bC$, and in particular has the same cohomological invariants. Jarvis has proved an algebraic analogue of Lemma 1.1 \cite[\S 3.3]{JarvisGeom}, showing that the coarse moduli space of $\mathfrak{S}_g^{1/r}$ has either one or two irreducible components depending on the parity of $r$.

Using his compactification, Jarvis has studied \cite{Jarvis} the Picard group of the stack $\mathfrak{S}_g^{1/r}$, and in particular produced algebraic line bundles of orders 4 and 3 when $r$ is divisible by 2 and 3 respectively \cite[Proposition 3.14]{Jarvis}. These are constructed in the same way as our classes $t^{a/r, b/r}$, insofar as that they are linear combinations of the $\lambda^{a/r}$ which are rationally trivial. Conjecture 4.4 of \cite{Jarvis} (also mentioned by Cornalba \cite[p. 31]{Cornalba}, \cite{Cornalba2}) is the presentation
$$\langle \lambda, \lambda^{1/2} \,\, | \,\, 4\lambda + 8\lambda^{1/2} \rangle$$
for the algebraic Picard group of $\gM_g^{1/2}[\epsilon]$, based on the following observations: $\lambda + 2\lambda^{1/2}$ is an element of order 4, $\lambda$ and $\lambda^{1/2}$ are elements of infinite order, and the Picard group is known to be $\bZ \oplus \bZ/4$ as an abstract group. However, our Example \ref{ex:rEq2} shows that in the torsion-free quotient of cohomology $\lambda$ is divisible by $4$ and $\lambda^{1/2}$ is divisible by $-2$. Furthermore, our Theorem \ref{thm:PicardMain} shows that
$$\Pic_{\alg}(\gM_g^{1/2}[\epsilon]) \lra H^2(\gM_g^{1/2}[\epsilon];\bZ)$$
is an isomorphism for $g \geq 9$. Hence the conjecture is incorrect, and the correct presentation is that of Example \ref{ex:rEq2}, where we interpret $\mu$ as a class constructed out of some square root of the algebraic line bundle $\lambda^{-1/2}$.

\section{Madsen--Weiss theory for $r$-Spin Riemann surfaces}

The main tool that allows us to say anything about the orbifolds $\gM_g^{1/r}$ is the development of a theory parallel to that of Madsen and Weiss \cite{MW} for $\gM_g$. Thus we have a good model for the homotopy type of the orbifold $\gM_g^{1/r}$, a homological stability theorem for these spaces, and an identification of the stable homology with the homology of a certain infinite loop space. The following outline of the theory is rather brief, and we suggest that the reader also consult \cite{MW, GMTW} for the analogous theory for $\gM_g$.

\subsection{A homotopical model for $\gM_g^{1/r}$}

Let $\gamma^r \to B\Spin^r(2)$ denote the universal $r$-Spin complex line bundle. Thus $B\Spin^r(2) \simeq BU(1)$ and the complex line bundle $\gamma^r$ is $r$ times the canonical bundle over $BU(1)$. Let $\Sigma_g$ be a closed oriented surface, and $\Bun(T\Sigma_g, \gamma^r)$ denote the space of bundle maps $T\Sigma_g \to \gamma^r$, i.e.\ fibrewise linear isomorphisms. We define the homotopy-theoretic moduli space to be the Borel construction
$$\mathcal{M}^{1/r}(\Sigma_g) := \Bun(T\Sigma_g, \gamma^r) \times_{\Diff^+(\Sigma_g)} \Emb(\Sigma_g, \bR^\infty).$$
By construction, homotopy classes of maps from a manifold into $\mathcal{M}^{1/r}(\Sigma_g)$ classify concordance classes of surface bundles with genus $g$ fibres, equipped with a complex line bundle $L$ on the total space and an isomorphism from $L^{\otimes r}$ to the vertical tangent bundle. There is a natural map $\mathcal{M}^{1/r}(\Sigma_g) \to \gM_g^{1/r}$ classifying the universal family, and we have immediately that

\begin{prop}
$\mathcal{M}^{1/r}(\Sigma_g)$ is the weak homotopy type of the orbifold $\gM_g^{1/r}$. In particular, they have isomorphic integral cohomology.
\end{prop}

\subsection{Homological stability}

In \cite[\S 2]{RWFramedPinMCG} we studied the spaces $\mathcal{M}^{1/r}(\Sigma_g)$ and their generalisations $\mathcal{M}^{1/r}(\Sigma_{g, b};\delta)$ to surfaces with boundary, where the $r$-Spin structure is required to satisfy a boundary condition $\delta$. The main theorem of that paper concerning these spaces is that they satisfy the hypotheses of \cite{R-WResolution} and hence exhibit homological stability. The implication of this statement is that

\begin{thm}
The homology of $\mathcal{M}^{1/r}(\Sigma_{g, b};\delta)$ is independent of $g$, $b$ and $\delta$ in degrees $5* \leq 2g-7$. More precisely, all stabilisation maps from this space obtained by gluing on an $r$-Spin surface along boundary components induce a homology isomorphism in this range of degrees.
\end{thm}

If we are discussing a certain homology group, we will use the term ``in the stable range" to mean ``for $g$ large enough to satisfy the above inequality".

\subsection{A universal approximation}\label{sec:UniversalApprox}

We have a complex line bundle $\gamma^r \to B\Spin^r(2)$, and we define the spectrum (in the sense of stable homotopy theory)
$$\MT{Spin^r}{2} := \mathbf{Th}(-\gamma^r \to B\Spin^r(2)),$$
that is, the Thom spectrum of the complement of the universal bundle over $B\Spin^r(2)$. Let $\Omega^\infty \MT{Spin^r}{2}$ denote the associated infinite loop space. By Pontrjagin--Thom theory the group $\pi_0(\Omega^\infty \MT{Spin^r}{2})$ is isomorphic to the group of oriented $r$-Spin surfaces up to cobordism (under disjoint union), but the cobordisms are required to be 3-manifolds whose tangent bundle is isomorphic to $\epsilon^1 \oplus L^{\otimes r}$ for some complex line bundle $L$, by an isomorphism which is standard at the boundaries. This may be seen by carrying out the Pontrjagin--Thom construction in this case, and we will not dwell on it. There is a natural homomorphism
$$\chi : \pi_0(\Omega^\infty \MT{Spin^r}{2}) \lra \bZ$$
given by sending an $r$-Spin surface to its Euler characteristic, which is well-defined over this cobordism relation by the Poincar{\'e}--Hopf theorem: the allowed cobordisms all admit a nowhere vanishing vector field which agrees with the inwards and outwards vector fields at the incoming and outgoing boundaries respectively. As $\Sigma_g$ admits an $r$-Spin structure precisely when $r \divides 2-2g$, we see that the image of $\chi$ is $2r\bZ$ if $r$ is odd and $r\bZ$ if $r$ is even. 

In \S \ref{sec:CalcPiZero}
we will show that
$$
\pi_0(\Omega^\infty \MT{Spin^r}{2}) \cong
\begin{cases}
\bZ & \text{$r$ odd}\\
\bZ \oplus \bZ/2 & \text{$r$ even}
\end{cases}
$$
as abstract groups, which implies that $\chi$ is injective when $r$ is odd and has kernel $\bZ/2$ when $r$ is even. When $r$ is even, there are maps
$$\MT{Spin^r}{2} \lra \MT{Spin^2}{2} \lra \Sigma^{-2} \mathbf{MSpin}$$
ending at the usual Spin bordism spectrum, and so a (surjective) homomorphism
$$\pi_0(\Omega^\infty \MT{Spin^r}{2}) \lra \Omega_2^{\Spin}(*) \cong \bZ/2$$
where the isomorphism is by sending a Spin surface to its Arf invariant. In total we obtain a canonical isomorphism
$$
\pi_0(\Omega^\infty \MT{Spin^r}{2}) \cong
\begin{cases}
2r\bZ & \text{$r$ odd}\\
r\bZ \oplus \bZ/2 & \text{$r$ even}
\end{cases}
$$
given by the Euler characteristic and the Arf invariant. We write $\Omega^\infty_{\chi, \bullet} \MT{Spin^r}{2}$ for the union of those (one or two) path components which map to $\chi$ under the Euler characteristic map.

Given an $r$-Spin surface bundle $(\Sigma_g \to E \overset{\pi}\to B, L \to E, \varphi : L^{\otimes r} \cong T_v)$ there is a Becker--Gottlieb \cite{BG} pretransfer 
$$\mathrm{prt} : \Sigma^\infty B_+ \lra \mathbf{Th}(-T_v \to E),$$
and a map $\tau^v : E \to B\Spin^r(2)$ classifying the vertical tangent bundle. Composing these gives a map
$$\alpha^\sharp := \mathbf{Th}(-\tau^v)\circ \mathrm{prt} : \Sigma^\infty B_+ \lra \MT{Spin^r}{2}$$
with adjoint
$$\alpha : B \lra \Omega^\infty_{2-2g, \bullet} \MT{Spin^r}{2}.$$
Performing this construction on the universal $r$-Spin bundle with genus $g$ fibres gives a comparison map
\begin{equation}\label{eq:ComparisonMap}
\alpha_g : \mathcal{M}^{1/r}(\Sigma_g) \lra \Omega^\infty_{2-2g, \bullet} \MT{Spin^r}{2}.
\end{equation}
The homological stability result and the methods of \cite{GMTW} or \cite{GR-W} imply that this map is an integral homology equivalence in degrees $5* \leq 2g-7$.

\subsection{The ``Mumford conjecture" for $\gM_g^{1/r}$}\label{sec:MumfordConjecture}

Theorem \ref{thm:MumfordConjecture} stated in the introduction will follow one we are able to compute the rational cohomology of the infinite loop space $\Omega^\infty_0 \MT{Spin^r}{2}$. This is easy: for any spectrum $\X$, the rational cohomology of the basepoint component of the associated infinite loop space, $\Omega^\infty_0 \X$, is the free graded commutative algebra on the rational vector space $H^{* > 0}_{\mathrm{spec}}(\X;\bQ)$ of positive degree elements in the spectrum cohomology of $\X$.

In our case, the forgetful map $\MT{Spin^r}{2} \to \MT{SO}{2}$ induces an isomorphism on rational cohomology (as we can compute the cohomology of both sides using the Thom isomorphism), and hence $\Omega^\infty_0 \MT{Spin^r}{2} \to \Omega^\infty_0\MT{SO}{2}$ also induces such an isomorphism. Theorem \ref{thm:MumfordConjecture} then follows from the known cohomology of $\Omega^\infty_0\MT{SO}{2}$, c.f.\ \cite{MW}.

\section{Computing the low-dimensional cohomology of $\gM_g^{1/r}$}

The purpose of this section is to give the following calculation of the first and second integral cohomology of the orbifold $\gM_g^{1/r}$, which establishes Theorem \ref{thm:LowDimHomology}.

\begin{thm}
There are isomorphisms
$$H^2(\gM_g^{1/r};\bZ) \cong \bZ \oplus
\begin{cases}
\bZ/4 & r \equiv 2 \,\, \mathrm{mod} \,\, 4\\
\bZ/8 & r \equiv 0 \,\, \mathrm{mod} \,\, 4\\
0 & \text{else}
\end{cases}
\oplus
\begin{cases}
\bZ/3 & r \equiv 0 \,\, \mathrm{mod} \,\, 3\\
0 & \text{else}
\end{cases}.
$$
and $H^1(\gM_g^{1/r};\bZ) =0$ as long as $g$ is in the stable range.
\end{thm}

For $g$ in the stable range there are isomorphisms
$$H^2(\gM_g^{1/r};\bZ) \cong H^2(\mathcal{M}^{1/r}(\Sigma_g);\bZ) \cong H^2(\Omega_0^\infty \MT{Spin^r}{2};\bZ).$$
The rank of $H^2(\Omega_0^\infty \MT{Spin^r}{2};\bZ)$ is the same as the rank of the second rational cohomology, which is one by Theorem \ref{thm:MumfordConjecture}. By the universal coefficient theorem and Hurewicz' theorem, we have
$$H^2(\Omega_0^\infty \MT{Spin^r}{2};\bZ) \cong \bZ \oplus H_1(\Omega^\infty_0 \MT{Spin^r}{2};\bZ) \cong \bZ \oplus \pi_1(\MT{Spin^r}{2})$$
and so the problem of computing the second cohomology of $\gM_g^{1/r}$ is reduced to the problem of computing a stable homotopy group. 

We approach this problem by observing that by the Thom isomorphism the cohomology $H^*(\MT{Spin^r}{2};\bZ)$ is a free module of rank one over $H^*(B\Spin^r(2);\bZ) = \bZ[x]$ on a generator $u_{-2} \in H^{-2}(\MT{Spin^r}{2};\bZ)$, and that at each prime the action of the Steenrod algebra on $H^*(\MT{Spin^r}{2};\bF_p)$ is determined by its action on $H^*(B\Spin^r(2);\bF_p)$ along with an identity
$$\mathcal{P}(u_{-2}) = f_p(x) \cdot u_{-2}$$
for some formal power series $f_p(x) \in H^*(B\Spin^r(2);\bF_p)$ which may be computed from characteristic classes of $\gamma^{r} \to B\Spin^r(2)$. In particular $f_2(x)$ is the total Stiefel--Whitney class of $-\gamma^{r}$, and $f_3(x)$ is the total Pontrjagin class of $-\gamma^{ r}$ reduced modulo 3.

\begin{defn}
Let $\X_r$ be the homotopy cofibre in the sequence
$$\MT{Spin^r}{2} \lra \MT{SO}{2} \lra \X_r.$$
\end{defn}

\begin{prop}\label{prop:XrIntegralCohomology}
The integral cohomology of $\X_r$ is given by $H^{even}(\X_r;\bZ)=0$ and $H^{2i+1}(\X_r;\bZ) = \bZ/r^{i+1}$ for $i \geq 0$.
\end{prop}
\begin{proof}
On cohomology, after applying the Thom isomoprhism, we have short exact sequences
$$0 \lra \bZ\{x^{i+1} \cdot u_{-2}\} \overset{\cdot r^{i+1}} \lra \bZ\{x^{i+1} \cdot u_{-2}\} \lra H^{2i+1}(\X_r;\bZ) \lra 0.$$
\end{proof}

We can now refine our discussion in \S \ref{sec:MumfordConjecture}: not only is the forgetful map a rational equivalence, it is an equivalence after inverting just $r$.

\begin{cor}
The spectrum $\X_r [\frac{1}{r}]$ is contractible, so the map of localised spectra
$$\MT{Spin^r}{2}\left[\frac{1}{r} \right] \lra \MT{SO}{2}\left[\frac{1}{r} \right]$$
is a homotopy equivalence. Hence $\Omega^\infty_0 \MT{Spin^r}{2} \lra \Omega^\infty_0 \MT{SO}{2}$ is a homology equivalence with any $\bZ[\tfrac{1}{r}]$-module coefficients. 
\end{cor}

The order of the group $\pi_1(\MT{Spin^r}{2})$ must then divide a power of $r$, as $\pi_1(\MT{SO}{2})=0$. The following lemma further controls the primes which can divide the order of this group.

\begin{lem}
For all primes $p \geq 5$, $H^*(\MT{Spin^r}{2};\bF_p)$ is a trivial module over the Steenrod algebra in degrees $* \leq 3$. Hence $\pi_1(\MT{Spin^r}{2})$ has no torsion of order $p$.
\end{lem}
\begin{proof}
Note that $H^*(\MT{Spin^r}{2};\bF_p)$ has no Bockstein operations, as it is supported in even degrees. Thus the shortest Steenrod operation is $\mathcal{P}^1$, which increases degrees by $2(p-1)$, so at least $8$ if $p \geq 5$. In particular $\mathcal{P}^1(u_{-2}) \in H^{\geq 6}(\MT{Spin^r}{2};\bF_p)$ is outside the range under consideration.

As it is a trivial module in this range, the $E^2$-page for the Adams spectral sequence is $\Sigma^{-2} M \oplus M \oplus \Sigma^2 M$ in this range, where $M := \mathrm{Ext}^{*,*}_{\mathcal{A}_p}(\bF_p,\bF_p)$. This consists of a $\bZ$-tower in degrees $-2$, $0$ and $2$, so $\pi_1$ has no $p$-torsion.
\end{proof}

Thus the group $\pi_1(\MT{Spin^r}{2})$ can have only 2- and 3-torsion, and only when 2 or 3 respectively divide $r$.

\begin{cor}
If $r$ is coprime to 6 then $H^2(\gM_g^{1/r};\bZ) \cong \bZ$ in the stable range.
\end{cor}

\begin{lem}
The image of $\pi_0(\MT{Spin^r}{2})$ in $\pi_0(\MT{SO}{2}) \cong \bZ$ is $r\bZ$ if $r$ is odd, and $r/2\bZ$ if $r$ is even.
\end{lem}
\begin{proof}
The identification $\pi_0(\MT{SO}{2}) \cong \bZ$ is by sending an oriented surface to half its Euler characteristic, so the problem is to determine which Euler characteristics can occur on $r$-Spin surfaces. The Euler characteristic may be identified with the divisibility of $c_1(T\Sigma_g)$, so if the surface has an $r$-Spin structure, $r | \chi$. Thus the minimal possible Euler characteristic of an r-Spin surface is $r$ if $r$ is even and $2r$ if $r$ is odd.
\end{proof}

\subsection{3-torsion}

At the prime 3 we have the identity
$$\mathcal{P}(u_{-2}) = p(\gamma_2^{\otimes r})\cdot u_{-2} = (1 - r^2 \cdot x^2) \cdot u_{-2},$$
where $p$ is the total Pontryjagin class. This determines the structure of the cohomology $H^*(\MT{Spin^r}{2};\bF_3)$ as a module over the Steenrod algebra. We are only interested in the case where 3 divides $r$, so $\mathcal{P}(u_{-2}) = u_{-2}$ and hence
$$H^*(\MT{Spin^r}{2};\bF_3) \cong \Sigma^{-2}H^{*}(BSO(2);\bF_3)$$
as modules over the Steenrod algebra. Thus the chart of the $E^2$-page of the Adams spectral sequence calculating the 3-primary homotopy of $\MT{Spin^r}{2}$ is as shown in Figure \ref{fig:p3Chart}. It immediately implies that $\pi_1(\MT{Spin^r}{2})_{(3)}$ is $\bZ/3$ if $3 \divides r$ and zero otherwise.
\begin{figure}[h]
\centering
\includegraphics[bb=0 0 179 133]{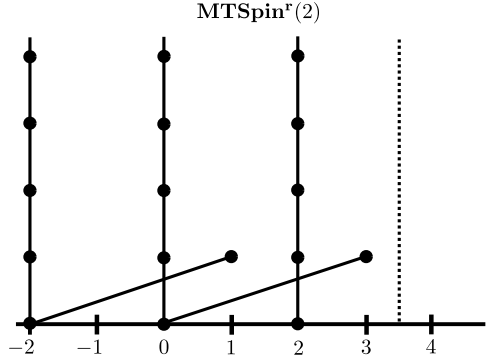}
\caption{Partial $E^2$-page of the Adams spectral sequence converging to the 3-primary homotopy groups of the spectrum $\MT{Spin^r}{2}$. The diagram is complete to the left of the dotted line.}
\label{fig:p3Chart}
\end{figure}

\subsection{2-torsion}

At the prime 2 we have the identity
$$Sq(u_{-2}) = w(\gamma_2^{\otimes r}) \cdot u_{-2} = (1 + r\cdot x) \cdot u_{-2}$$
and we are only interested in the case where 2 divides $r$ so we obtain $Sq(u_{-2}) = u_{-2}$ and hence
$$H^*(\MT{Spin^r}{2};\bF_2) \cong \Sigma^{-2}H^{*}(BSO(2);\bF_2)$$
as modules over the Steenrod algebra. Thus the chart of the $E^2$-page of the Adams spectral sequence calculating the 2-primary homotopy of $\MT{Spin^r}{2}$ is as shown in Figure \ref{fig:p2Chart}.
\begin{figure}[h]
\centering
\includegraphics[bb=0 0 179 133]{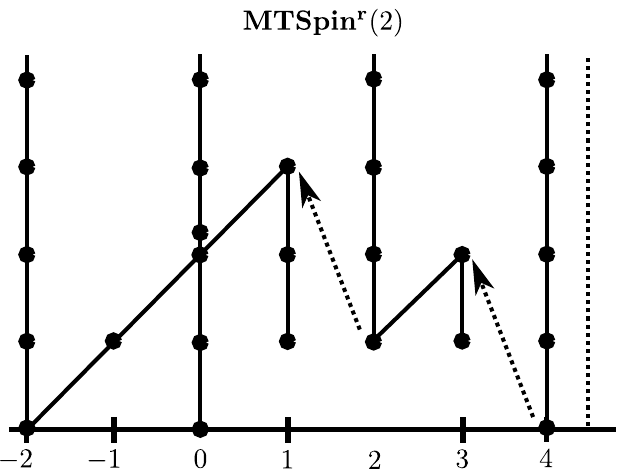}
\includegraphics[bb=0 0 178 133]{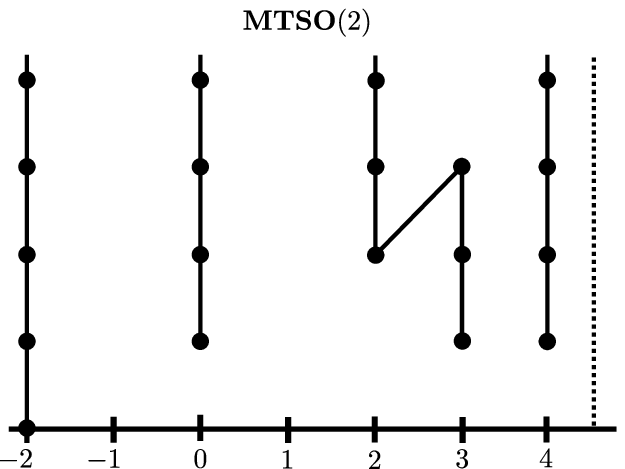}
\caption{Partial $E^2$-page of the Adams spectral sequences converging to the 2-primary homotopy groups of the spectra $\MT{Spin^r}{2}$ and $\MT{SO}{2}$. The diagrams are complete to the left of the dotted line, and the dotted arrows show the only possible differentials in this range.}
\label{fig:p2Chart}
\end{figure}
There is an ambiguity in $\pi_1(\MT{Spin^r}{2})_{(2)}$ given by the possible differential, which we must resolve. In order to do this we also compute the $E^2$-page of the Adams spectral sequence for the cofibre $\X_r$ at the prime 2 (when $r$ is even). This depends on whether $r$ is divisible by 4 or not.

\begin{prop}
The cohomology groups $H^*(\X_r;\bF_2)$ are a copy of $\bF_2$ in each degree $i \geq 0$, generated by an element $a_i$. As a module over the Steenrod algebra, up to degree 4 the only non-trivial operations are
$$Sq^2(a_2) = a_4, \,\,Sq^2(a_1) = a_3,$$
and if $r \equiv 2 \,\,\mathrm{mod}\,\, 4$ then $Sq^1(a_0)=a_1$.
\end{prop}
\begin{proof}
The identification of the cohomology groups follows from Proposition \ref{prop:XrIntegralCohomology} and the Universal Coefficient Theorem. The identification of the Steenrod operations follows from the known operations in $\MT{SO}{2}$ and $\MT{Spin^r}{2}$ in this range.
\end{proof}
Thus if $r \equiv 2 \,\,\mathrm{mod}\,\, 4$ then $H^*(\X_r;\bF_2)$ is the module
$$(\bF_2 \overset{Sq^1}\lra \Sigma^{1}\bF_2 \overset{Sq^2}\lra \Sigma^{3}\bF_2) \oplus (\Sigma^2\bF_2 \overset{Sq^2}\lra \Sigma^4 \bF_2)$$
in degrees $\leq 4$, and if $r \equiv 0 \,\,\mathrm{mod}\,\, 4$ it is the module
$$\bF_2 \oplus (\Sigma^{1}\bF_2\overset{Sq^2}\lra \Sigma^3 \bF_2) \oplus (\Sigma^2\bF_2 \overset{Sq^2}\lra \Sigma^4 \bF_2)$$
in this range. The $E^2$-pages of the corresponding Adams spectral sequences split accordingly in this range, and each elementary module has a well-known Adams $E^2$-page.

\subsection{Case 1}
Let us first treat the case $r \equiv 2 \,\,\mathrm{mod}\,\, 4$, where a chart for the Adams $E^2$-page for $\X_r$ is shown in Figure \ref{fig:Xr2mod4Chart}.
\begin{figure}[h]
\centering
\includegraphics[bb=0 0 140 133]{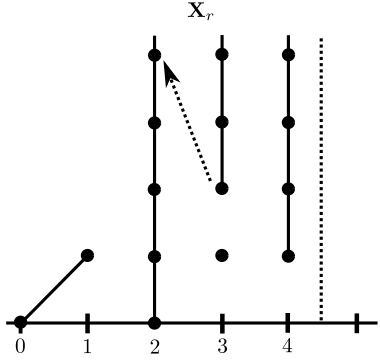}
\caption{Partial $E^2$-page of the Adams spectral sequences converging to the 2-primary homotopy groups of the spectrum $\gX_r$, when $r \equiv 2 \,\, \mathrm{mod} \,\, 4$. The diagram is complete to the left of the dotted line. There is a differential into the total degree 2 column: the dotted arrow shows the shortest possible such differential.}
\label{fig:Xr2mod4Chart}
\end{figure}
As $\X_r$ is $r$-torsion, the two $\bZ$-towers in degrees 2 and 3 must kill each other, so that $\pi_2(\X_r)_{(2)} = \bZ/2^k$ for some $k \geq 4$. The long exact sequence in 2-local homotopy gives the exact sequences
$$\cdots \lra \pi_4(\X_r)_{(2)} \lra \pi_3(\MT{Spin^r}{2})_{(2)} \lra \bZ/8 \lra \bZ/2 \lra 0,$$
and
$$0 \lra \bZ_{(2)} \lra \bZ_{(2)} \lra \bZ/2^k \lra \pi_1(\MT{Spin^r}{2})_{(2)} \lra 0.$$
The first sequence implies that $\pi_3(\MT{Spin^r}{2})_{(2)}$ has order at least 4, and so the differential entering the total degree 3 column is zero, and $\pi_3(\MT{Spin^r}{2})_{(2)} \cong \bZ/4$.

Suppose now that the differential entering the total degree 1 column is zero: then the Adams spectral sequence for $\MT{Spin^r}{2}$ collapses in the range we have drawn it, and in particular we may read off its $\pi_*(\gS)$-module structure. By the short exact sequence
$$0 \lra \pi_3(\MT{Spin^r}{2})_{(2)} \lra \pi_3(\MT{SO}{2})_{(2)} \lra \bZ/2 \lra 0$$
and comparing the (collapsing) charts for $\MT{Spin^r}{2}$ and $\MT{SO}{2}$, we see that $\bZ_{(2)} \cong \pi_2(\MT{Spin^r}{2})_{(2)} \to \pi_2(\MT{SO}{2})_{(2)} \cong \bZ_{(2)}$ must be an isomorphism. Thus
$$\bZ/2^{k \geq 4} \cong \pi_2(\gX_r)_{(2)} \cong  \pi_1(\MT{Spin^r}{2})_{(2)} \cong \bZ/8,$$
which is a contradiction. Thus when $r \equiv 2\,\, \mathrm{mod}\,\, 4$ the differential entering the degree 1 column is non-trivial, and so
$$\pi_1(\MT{Spin^r}{2})_{(2)} \cong \bZ/4.$$

\subsection{Case 2} Let us now treat the case $r \equiv 0 \,\, \mathrm{mod} \,\, 4$, where a chart for the Adams $E^2$-page for $\X_r$ is shown in Figure \ref{fig:Xr0mod4Chart}.
\begin{figure}[h]
\centering
\includegraphics[bb=0 0 140 133]{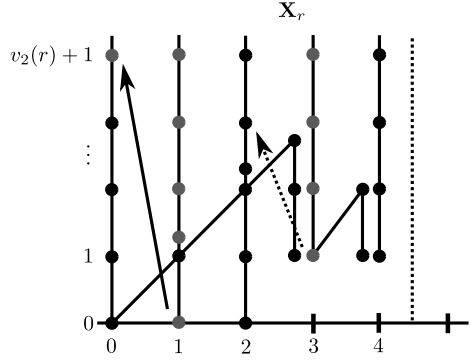}
\caption{Partial $E^2$-page of the Adams spectral sequences converging to the 2-primary homotopy groups of the spectrum $\gX_r$, when $r \equiv 0 \,\, \mathrm{mod} \,\, 4$. The diagram is complete to the left of the dotted line. The grey dots show groups which definitely die. There is a differential into the total degree 2 column: the dotted arrow shows the shortest possible such differential.}
\label{fig:Xr0mod4Chart}
\end{figure}
As $\X_r$ is $r$-torsion, the $\bZ$-towers must kill each other. As $\pi_0(\gX_r)$ fits into an exact sequence
$$\cdots \to\bZ \oplus \bZ/2 \overset{\cdot r/2}\to \bZ \to \pi_0(\gX_r) \to \bZ/2 \to 0$$
it is finite of order $r$, so its 2-component is finite of order $2^{v_2(r)}$. From the chart we see that it is cyclic of this order, so $\pi_0(\gX_r) \cong \bZ/r$, and we have the computation
$$\pi_*(\X_r)_{(2)} = 
\begin{cases}
\bZ/r \otimes \bZ_{(2)} & *=0\\
\bZ/2 & *=1\\
\bZ/2 \oplus \bZ/2^k & *=2\quad k \geq 3
\end{cases}
$$
in degrees up to 2. Consider the portion of the long exact sequence in 2-local homotopy,
$$\cdots \to \pi_3(\gX_r)_{(2)} \overset{0}\to \bZ_{(2)} \to \bZ_{(2)} \overset{f}\to \bZ/2 \oplus \bZ/2^{k \geq 3} \to \pi_1(\MT{Spin^r}{2})_{(2)} \to 0.$$
The map $f$ is onto the first factor, so suppose it sends 1 to $(1, 2^N)$. Then
$$\pi_1(\MT{Spin^r}{2})_{(2)} \cong \bZ/2^{\mathrm{min}(k, N+1)}.$$
The generator of $\pi_2(\MT{SO}{2})_{(2)}$ is in Adams filtration 2, so its image has filtration at least 2. Thus $N \geq 2$, and so $\mathrm{min}(k, N+1) \geq 3$, but $\pi_1(\MT{Spin^r}{2})_{(2)}$ is also at most $\bZ/8$, so $\mathrm{min}(k, N+1) = 3$, and so $N=2$. Thus
$$\pi_1(\MT{Spin^r}{2})_{(2)} \cong \bZ/8,$$
and there is no differential.

\subsection{$\pi_0(\MT{Spin^r}{2})$}\label{sec:CalcPiZero}

As we have used it in \S \ref{sec:UniversalApprox}, we remark that Figures \ref{fig:p3Chart} and \ref{fig:p2Chart} show that
$$
\pi_0(\MT{Spin^r}{2}) \cong
\begin{cases}
\bZ & \text{$r$ odd}\\
\bZ \oplus \bZ/2 & \text{$r$ even}
\end{cases}
$$
as abstract groups.

\subsection{Induced maps}

\begin{prop}
The map $\pi_2(\MT{Spin^r}{2}) = \bZ \to \pi_2(\MT{SO}{2}) = \bZ$ is given by multiplication by $ r^2U_r/12$, where
$$U_r = \begin{cases}
2 & 12 \divides r\\
4 & 4 \ndivides r, 3 \divides r\\
6 & 4 \divides r, 3 \ndivides r\\
12 & 4 \ndivides r, 3 \ndivides r.
\end{cases}$$
\end{prop}
\begin{proof}
There is a commutative square
\begin{diagram}
\pi_2(\MT{Spin^r}{2}) = \bZ &\rTo^{\cdot U_r}& H_2(\MT{Spin^r}{2};\bZ) = \bZ\\
\dTo & & \dTo^{\cdot r^2}\\
\pi_2(\MT{SO}{2}) = \bZ &\rTo^{\cdot 12}& H_2(\MT{SO}{2};\bZ) = \bZ,
\end{diagram}
where the Hurewicz map $U_r$ may be determined by studying the Atiyah--Hirzebruch spectral sequence
$$E_{p,q}^2 = H_p(\MT{Spin^r}{2};\pi_q(\gS)) \Rightarrow \pi_{p+q}(\MT{Spin^r}{2}),$$
where it occurs as an edge homomorphism. There is a unique pattern of differentials out of $E^2_{2,0}$ consistent with the known homotopy groups of $\MT{Spin^r}{2}$, and this determines the index of the image of the Hurewicz map.
\end{proof}

\section{Identifying classes in $H^2(\gM_g^{1/r};\bZ)$}

In the introduction we defined cohomology classes $\kappa_1^{a/r}$ and $\lambda^{a/r}$ in $H^2(\gM_g^{1/r};\bZ)$, similar in spirit to the classes $\kappa_1$ and $\lambda$ defined in the cohomology of $\gM_g$. When $r$ is even we have defined a further class $\mu$, which satisfies the equation $2\mu = \lambda^{-1/2} + 12\lambda^{1/2}$. All of these classes may in fact be defined on the infinite loop space $\Omega^\infty\MT{Spin^r}{2}$, as they all come from fibre-integration of a stable characteristic class associated to the vertical tangent bundle. In this section we will explain how to construct them in this manner.

\begin{defn}
For any generalised cohomology theory $E$ and any spectrum $\gX$, there is a \emph{cohomology suspension} map
$$\sigma^* : E^*(\gX) \lra E^*(\Omega^\infty \gX),$$
given by applying $E^*$ to the evaluation maps $\Sigma^n \Omega^n X_n \to X_n$ and taking a limit over $n$. 
\end{defn}

We write $\gamma^{r} \to B\Spin^r(2)$ for the tautological bundle, and $L$ for its canonical $r$-th root. The Thom isomorphism in spectrum cohomology gives
\begin{eqnarray*}
H^*(B\Spin^r(2);\bZ) &\lra& H^{*-2}(\MT{Spin^r}{2};\bZ)\\
x &\longmapsto& x \cdot u_{-2},
\end{eqnarray*}
and hence we may define $\kappa_1^{a/r} \in H^2(\Omega^\infty \MT{Spin^r}{2};\bZ)$ as $\sigma^*(c_1(L^a)^{2} \cdot u_{-2})$.

The virtual bundle $-\gamma^{r} \to B\Spin^r(2)$ is complex, and hence oriented in complex K-theory. Thus there is a Thom isomorphism in spectrum K-theory
\begin{eqnarray*}
K^0(B\Spin^r(2)) &\lra& K^{0}(\MT{Spin^r}{2})\\
x &\longmapsto& x \cdot \lambda_{-\gamma^{r}},
\end{eqnarray*}
and hence an element $\sigma^*(L^{\otimes a} \cdot \lambda_{-\gamma^{r}}) \in K^0(\Omega^\infty\MT{Spin^r}{2})$. We may define $\lambda^{-a/r}$ to be the first Chern class of this virtual bundle.

If $r$ is even, there is an isomorphism $(L^{\otimes r/2})^{\otimes 2} \cong \gamma^{r}$ and so $\gamma^{r}$ has a canonical Spin structure and hence is oriented in real K-theory. Thus there is a Thom isomorphism in spectrum KO-theory
\begin{eqnarray*}
KO^0(B\Spin^r(2)) &\lra& KO^{-2}(\MT{Spin^r}{2})\\
x &\mapsto& x \cdot \lambda_{-\gamma^{r}},
\end{eqnarray*}
and hence an element $\xi := \sigma^*(1 \cdot \lambda_{-\gamma^{r}}) \in KO^{-2}(\Omega^\infty\MT{Spin^r}{2})$. Elements in $KO^{-2}$ are represented by complex vector bundles with a trivialisation of the underlying real vector bundle, and such bundles have a canonical choice of half the first Chern class. We define $\mu := \tfrac{c_1}{2}(\xi) + 6\lambda^{1/2}$.

\subsection{Divisibility of classes in the torsion-free quotient}

\begin{proof}[Proof of Theorem \ref{thm:TorsionFreeDivisibility}]
Let us consider the fibration sequence of connected infinite loop spaces,
$$\Omega^{\infty+1}_0 \gX_r \lra \Omega^\infty_0 \MT{Spin^r}{2} \lra \Omega^\infty_0 \MT{SO}{2}.$$
Considering the Serre spectral sequence in cohomology, we see that there is a short exact sequence
$$0 \to H^2(\Omega^\infty_0\MT{SO}{2};\bZ) \to H^2(\Omega^\infty_0\MT{Spin^r}{2};\bZ) \to H^2(\Omega^{\infty+1}_0 \gX_r;\bZ) \to 0,$$
and furthermore, restricting to the torsion subgroup, the map
$$\mathrm{Tors} \, H^2(\Omega^\infty_0\MT{Spin^r}{2};\bZ) \lra H^2(\Omega^{\infty+1}_0 \gX_r;\bZ)$$
is Pontrjagin dual to the connecting homomorphism
$$H_1(\Omega^{\infty+1}_0 \gX_r) \cong \pi_2(\gX_r) \to H_1(\Omega^\infty_0\MT{Spin^r}{2}) \cong \pi_1(\MT{Spin^r}{2}).$$
The long exact sequence of homotopy groups
$$\cdots \to \pi_2(\MT{Spin^r}{2}) \overset{r^2U_r/12}\to \pi_2(\MT{SO}{2})  \to \pi_2(\gX_r) \to \pi_1(\MT{Spin^r}{2}) \to 0$$
shows that this connecting homomorphism is surjective with kernel $\bZ/(r^2U_r/12)$. Thus the Pontrjagin dual map is injective with cokernel $\bZ/(r^2U_r/12)$, so there is an exact sequence
$$0 \to H^2(\Omega^\infty_0\MT{SO}{2};\bZ) \to H^2(\Omega^\infty_0\MT{Spin^r}{2};\bZ)/\mathrm{torsion} \to \bZ/(r^2U_r/12) \to 0,$$
and hence $\lambda$ is divisible by precisely $r^2U_r/12$ in the torsion-free quotient.

The remaining divisibilities follow once we establish the rational proportionalities between the classes $\{\kappa_1^{a/r}, \lambda^{a/r}, \mu\}$ and $\lambda$. Firstly $\kappa_1^{a/r} = \tfrac{1}{a^2} \kappa_1 = \tfrac{12}{a^2} \lambda$. In order to relate $\lambda^{-a/r}$ to the other classes, we apply the Chern character:
$$\mathrm{ch}(\pi_!(L^a)) = \pi_!(\mathrm{ch}(L^a) \cdot \Td(T_v)) = \pi_!\left(e^{ac_1(L)}\cdot \frac{c_1(T_v)}{1-e^{-c_1(T_v)}}\right)$$
and note that $c_1(T_v) = r \cdot c_1(L)$. Thus
\begin{equation}\label{eq:ProportionalityLambdaClasses}
\lambda^{-a/r} = \frac{r^2 +6ar + 6a^2}{r^2}\lambda \in H^2(\gM_g^{1/r};\bQ).
\end{equation}
Finally we have that $2\mu = \lambda^{-1/2} + 12\lambda^{1/2}$. If we let $g$ be a generator of the torsion-free quotient such that $\lambda = \tfrac{r^2U_r}{12}g$, then 
\begin{eqnarray*}
\kappa_1^{a/r} &=& a^2U_rg\\
\lambda^{a/r} &=& \frac{U_r}{12}(r^2 - 6ar+6a^2)g\\
\mu &=& -\frac{U_r}{48}r^2 g \quad\quad \text{when it is defined.}
\end{eqnarray*}
\end{proof}

\subsection{Torsion classes}

In the introduction we have defined torsion classes $t^{a/r, b/r}$, $t^{a/r}$ and $t$ in the cohomology of $\gM_g^{1/r}$ as certain linear combinations of $\lambda^{a/r}$, $\kappa_1^{1/r}$ and $\mu$ which vanish in the torsion-free quotient. We recall: if $U_{a/r, b/r} = \gcd(r^2-6ar+6a^2, r^2-6br+6b^2)$ then
$$t^{a/r, b/r} := \frac{r^2-6br+6b^2}{U_{a/r, b/r}} \lambda^{a/r} - \frac{r^2-6ar+6a^2}{U_{a/r, b/r}}\lambda^{b/r}$$
and
$$t^{a/r} := \frac{12}{\gcd(12, r^2-6ar+6a^2)}\lambda^{a/r} - \frac{r^2-6ar+6a^2}{\gcd(12, r^2-6ar+6a^2)} \kappa_1^{1/r}.$$
are torsion. When $r$ is even, there is also the torsion class
$$t := \frac{48}{\gcd(r^2, 48)}\mu + \frac{r^2}{\gcd(r^2,48)}\kappa_1^{1/r}.$$

Although we have calculated $H^2(\gM_g^{1/r}[\epsilon];\bZ)$ as an abstract group, we do not yet have a way of determining whether the torsion classes we have constructed are non-zero or not. The main result of this section is a detection theorem which will allow us to identify torsion classes.

\begin{lem}\label{lem:Detecting}
There is a natural homomorphism
$$i^*: H^2(\Omega^\infty_0 \MT{Spin^r}{2};\bZ) \lra H^2(\Omega^2_0 Q(S^0);\bZ)$$
which is injective when restricted to the torsion subgroup.
\end{lem}
\begin{proof}
Consider the map of spectra
$$i: \gS^{-2} \lra \MT{Spin^r}{2}$$
given by the inclusion of the $-2$-cell. Consulting the Adams spectral sequence charts for $\MT{Spin^r}{2}$ in all possible cases, we see that $\pi_1(\MT{Spin^r}{2})$ is obtained entirely from $\pi_{-2}(\MT{Spin^r}{2})$ via the $\pi_*(\gS)$-module structure. This implies that $\pi_1(\gS^{-2}) \to \pi_1(\MT{Spin^r}{2})$ is surjective, so taking infinite loop spaces the homomorphism $H_1(\Omega^2_0 Q(S^0);\bZ) \to H_1(\Omega^\infty_0\MT{Spin^r}{2};\bZ)$ is surjective, and the claim follows by Pontrjagin duality.
\end{proof}

Thus we may test the non-triviality of torsion elements by computing them on $\Omega^2_0 Q(S^0)$.

\begin{prop}\label{prop:ClassesOnQS0}
The classes $i^*\lambda^{a/r}$ are all equal, and so in particular all equal to $i^*\lambda$; this class is twice a generator of $H^2(\Omega^2_0QS^0;\bZ) \cong \bZ/24$. When $\mu$ is defined, $i^*\mu$ is a generator of $H^2(\Omega^2_0QS^0;\bZ) \cong \bZ/24$. The classes $i^*\kappa_1^{a/r}$ are all zero.
\end{prop}
\begin{proof}
The class $i^*\lambda^{a/r}$ is the first Chern class of the K-theory class 
$$\gS^{-2} \overset{i}\lra \MT{Spin^r}{2} \overset{L^a \cdot \lambda_{-\gamma^{r}}}\lra \mathbf{K},$$
which may be obtained using the Thom isomorphism in K-theory of the virtual bundle $-\bC \to *$ applied to the pullback of the line bundle $L^a \to B\Spin^r(2)$ to a point: this is the trivial line bundle over a point for all $a$, which proves the first assertion. More precisely, considered as an element of $\pi_{-2}(\gK)$ this is $\beta^{-1}$, the inverse of the Bott element. In Appendix \ref{sec:Appendix:Calculation}, we compute the effect of the map
$$H^2(BU;\bZ) \cong_{Bott} H^2(\Omega^2_0 (\bZ \times BU);\bZ) \overset{\Omega^2_0(unit)}\lra H^2(\Omega^2_0 QS^0;\bZ) \cong \bZ/24$$
and show it sends the first Chern class to twice a generator.

The statement about $\mu$ is immediate, as $2\mu = \lambda^{-1/2} + 12\lambda^{1/2}$ and $12i^*(\lambda^{1/2})=0$ so $2i^*(\mu) = i^*(\lambda^{-1/2})$ is twice a generator, so $i^*(\mu)$ is a generator. The statement about $\kappa_1^{a/r}$ follows because $i^*\sigma^*(c_1(L^a)^2 \cdot u_{-2}) = \sigma^*(i^*(c_1(L^a)^2\cdot u_{-2}))$ but $i^*(c_1(L^a)^2\cdot u_{-2}) \in H^2(\gS^{-2};\bZ)=0$.
\end{proof}

\begin{proof}[Proof of Theorem \ref{thm:TorsionDetection}]
When $g \geq 9$, the canonical map
$$\alpha_g^* : H^2(\Omega_{2-2g, \epsilon}^\infty \MT{Spin^r}{2};\bZ) \lra H^2(\mathcal{M}^{1/r}(\Sigma_g)[\epsilon];\bZ) \cong H^2(\gM_g^{1/r}[\epsilon];\bZ)$$
is an isomorphism, and the equivalence $\Omega_{2-2g, \epsilon}^\infty \MT{Spin^r}{2} \simeq \Omega_0^\infty \MT{Spin^r}{2}$ is canonical up to homotopy too. Thus the map of Lemma \ref{lem:Detecting} provides a canonical map
$$\bar{\varphi} : H^2(\gM_g^{1/r}[\epsilon];\bZ) \lra H^2(\Omega^2_0 QS^0 ;\bZ)$$
which is injective on the torsion subgroup. By the universal coefficient theorem, this last group is naturally isomorphic to $\mathrm{Ext}(\pi_3^s, \bZ)$, which is un-naturally isomorphic to $\bZ/24$.

When $r=2$ the class $\bar{\varphi}(\mu) \in H^2(\Omega^2_0 QS^0 ;\bZ)$ is a generator by Proposition \ref{prop:ClassesOnQS0} and so describes an isomorphism $H^2(\Omega^2_0 QS^0 ;\bZ) \cong \bZ/24$ under which $\bar{\varphi}(\mu)$ goes to 1. Let us use this as our standard identification of $H^2(\Omega^2_0 QS^0 ;\bZ)$ with $\bZ/24$, giving a map
$$\varphi : H^2(\gM_g^{1/r}[\epsilon];\bZ) \lra \bZ/24.$$
It remains to establish the effect of this map on the elements $\lambda^{a/r}$, $\kappa_1^{a/r}$ and $\mu$ when it is defined. By Proposition \ref{prop:ClassesOnQS0} all the $\kappa_1^{a/r}$ are sent to zero, and all the $\lambda^{a/r}$ are sent to the same thing. As $2\mu = \lambda^{-1/2}+12\lambda^{1/2}$ and $12\varphi(\lambda^{1/2})=0$, it is enough to check where $\mu$ is sent when $r$ is even, and where $\lambda$ is sent when $r$ is odd. For this we make use of the following diagram:
\begin{diagram}
\Omega^2_0 QS^0 & \rTo & \Omega^\infty_0\MT{Spin^{2r}}{2} & \rTo & \Omega^\infty_0\MT{Spin^{r}}{2}\\
& \rdTo & \dTo\\
& & \Omega^\infty_0\MT{Spin^{2}}{2}
\end{diagram}
By definition of the isomorphism $H^2(\Omega^2_0 QS^0 ;\bZ) \cong \bZ/24$, the diagonal map sends $\mu$ to $1 \in \bZ/24$ and hence $\lambda^{a/r}$ to $2 \in \bZ/24$. By naturality the same is true of the top left map, as $\mu$ is pulled back to a class of the same name here. Naturality of $\mu$ and $\lambda^{a/r}$ with respect to the top right map establishes the theorem.
\end{proof}

Using the formul{\ae} for the torsion elements $t^{a/r}$ and $t$, and the fact that we are able to compute them pulled back to $\Omega^2_0 QS^0$, we establish Corollary \ref{cor:TorsionGenerators}.

\begin{proof}[Proof of Corollary \ref{cor:TorsionGenerators}]
We use the formul{\ae} of Theorem \ref{thm:TorsionDetection} for $\varphi$.
Suppose $r$ is odd, and $3 \divides r$ (as otherwise there is nothing to show). We have that $\varphi(t^{a/r})$ is $24/\gcd(12, r^2-6ar+6a^2)$ in $\bZ/24$. The $\gcd$ is precisely 3 in this case, so we get $8$, and $\varphi$ is injective when restricted to the torsion subgroup, so we are done.

Suppose $r \equiv 2 \,\,\mathrm{mod}\,\, 4$. We have that $\varphi(t^{0/r})$ is $24/\gcd(12, r^2)$ in $\bZ/24$, which is $2$ if $3 \divides r$ or $6$ if $3 \ndivides r$. As $\varphi$ is injective when restricted to the torsion subgroup, we are done.

Suppose $r \equiv 0 \,\,\mathrm{mod}\,\, 4$. We have that $\varphi(t)$ is $48 / \gcd(48,r^2)$ in $\bZ/24$, which is $1$ if $3 \divides r$ or $3$ if $3 \ndivides r$. As $\varphi$ is injective when restricted to the torsion subgroup, we are done.
\end{proof}

\subsection{Twists of $r$-Spin structures}

Let $\pi : E \to B$ be a $r$-Spin $\Sigma_g$-bundle, that is, an oriented surface bundle of genus $g$ with a complex line bundle $L \to E$ and an isomorphism $\varphi : L^{\otimes r} \cong T_vE$. Suppose that $D \to B$ is a complex line bundle equipped with an isomorphism $D^{\otimes r} \cong \epsilon^1$. Such bundles are classified up to isomorphism by elements of $H^1(B;\bZ/r)$, and we also write $D$ for the cohomology class it represents. We can produce a new $r$-Spin structure on $\pi$ by $L_D := L \otimes \pi^*D$, as then $L_D^{\otimes r} \cong L^{\otimes r} \otimes \pi^*D^{\otimes r} \cong T_vE \otimes \epsilon^1$.

We give here a formularium for computing the characteristic classes $\kappa_1^{a/r}$, $\lambda^{a/r}$ and $\mu$ of the $r$-Spin structure $L_D$ in terms of those of $L$ and the class $D \in H^1(B;\bZ/r)$.

\begin{thm}\label{thm:Twists}
There are equalities
\begin{eqnarray*}
\kappa_1^{a/r}(L_D) &=& \kappa_1^{a/r}(L) + \frac{2a^2\chi(\Sigma_g)}{r}\beta(D)\\
\lambda^{-a/r}(L_D) &=& \lambda^{-a/r}(L)\\
\mu(L_D) &=& \mu(L) + \mathrm{Arf}(L)\tfrac{r}{2} \beta(D)
\end{eqnarray*}
in $H^2(B;\bZ)$.
\end{thm}
\begin{proof}
The first equation follows easily from the definition $L_D = L \otimes \pi^*D$ and the fact that $c_1(D)$ is the Bockstein $\beta(D)$ when we consider $D$ as an element of $H^1(B;\bZ/r)$: so $c_1(L_D) = c_1(L) + \beta(D)$.

For the second equation, we calculate
$$\lambda^{-a/r}(L_D) = c_1(\pi_!(L_D^a)) = c_1(\pi_!(L^a \otimes \pi^*(D^a))) = c_1(\pi_!(L^a) \otimes \pi^*(D^a) )$$
and bear in mind that $\dim(\pi_!(L^a)) = (a + \tfrac{r}{2})\chi(\Sigma_g)$, and that $r \divides \chi(\Sigma_g)$.

The last equation is more complicated, and follows from the equality
$$\pi_!^{KO}(1)_{L_D^{\otimes r/2}} = D^{\otimes r/2} \otimes_\bC \pi_!^{KO}(1)_{L^{\otimes r/2}} \in KO^{-2}(B),$$
where $\pi_!^{KO}(-)_X$ denotes the pushforward in real $K$-theory defined using the Spin structure $X$. This equality comes from the formula for the $KO$-theory Thom class given by the Spin structure $L_D^{\otimes r/2}$, in terms of the Thom class for the Spin structure $L^{\otimes r/2}$. This shows that $\tfrac{c_1}{2}(\xi_D) = \frac{c_1}{2}(\xi) + \mathrm{Arf}(L)\tfrac{r}{2} \beta(D)$, noting that $\mathrm{dim}_\bC(\pi_!^{KO}(1)_{L^{\otimes r/2}}) \equiv \mathrm{Arf}(L) \,\, \mathrm{mod} \,\, 2$. The equation in the statement of the theorem is obtained by noting that $\lambda^{1/2}$ is unchanged under twisting the $r$-Spin structure.
\end{proof}

\section{Computations of the homology of $\widetilde{\gM}_g^{1/r}$}\label{sec:ThetaChar}

We can use the extension (\ref{eq:GroupExt}) and the characteristic class description of the elements of $H^2(\gM_g^{1/r};\bZ)$ to compute the low-dimensional cohomology of $\widetilde{\gM}_g^{1/r}$. In order to do so, we wish to compute the effect on first homology of the inclusion of the fibre in the fibration
\begin{equation}\label{eq:ExtensionFibration}
B\bZ/r \lra \mathcal{M}_g^{1/r}[\epsilon] \lra \widetilde{\mathcal{M}}_g^{1/r}[\epsilon].
\end{equation}
The map $H_1(B\bZ/r;\bZ) \to H_1(\mathcal{M}_g^{1/r}[\epsilon];\bZ)$ is Pontrjagin dual to the map
\begin{equation}\label{eq:TorsMap}
\mathrm{Tors} \, H^2(\mathcal{M}_g^{1/r}[\epsilon];\bZ) \lra H^2(B\bZ/r;\bZ) \cong \bZ/r,
\end{equation}
and so it suffices to compute the effect of this map.

By Corollary \ref{cor:TorsionGenerators}, we have generators for the torsion subgroup of $H^2(\mathcal{M}_g^{1/r}[\epsilon];\bZ)$ in terms of the elements $t^{a/r}$ and $t$. In order to calculate the images of these elements under the above map, we must compute the characteristic classes $\kappa_1^{a/r}$, $\lambda^{a/r}$ and $\mu$ on the $r$-Spin surface bundle classified by the map $B\bZ/r \to \mathcal{M}_g^{1/r}[\epsilon]$. The bundle this map classifies is, by definition, the twist of the trivial $r$-Spin bundle classified by the constant map $B\bZ/r \to \mathcal{M}_g^{1/r}[\epsilon]$ by the tautological class $x \in H^1(B\bZ/r;\bZ/r)$. Theorem \ref{thm:Twists} tells us how to compute the characteristic classes of this bundle: they are
\begin{eqnarray*}
\kappa_1^{a/r} &=& \frac{2a^2\chi(\Sigma_g)}{r}\beta(x)\\
\lambda^{a/r} &=& 0\\
\mu &=& \epsilon\tfrac{r}{2}\beta(x)
\end{eqnarray*}
when the class $\mu$ makes sense (i.e.\ when $r$ is even). Carrying out these computations using Corollary \ref{cor:TorsionGenerators} leads to the following description of the map (\ref{eq:TorsMap}).

\begin{prop}\label{prop:EffectOnH1}
If $r \equiv 2 \,\,\mathrm{mod}\,\, 4$ the map (\ref{eq:TorsMap}) is zero. If $r \equiv 0 \,\,\mathrm{mod}\,\, 4$ the map (\ref{eq:TorsMap}) has image $\frac{r(12\epsilon+\chi(\Sigma_g))}{8\gcd(r,3)}\bZ/r$. If $r$ is odd the map (\ref{eq:TorsMap}) is zero.
\end{prop}
\begin{proof}
Recall the statement of Corollary \ref{cor:TorsionGenerators}: if $r$ is odd, any $t^{a/r}$ generates the torsion subgroup; if $r \equiv 2 \,\,\mathrm{mod}\,\, 4$ then $t^{0/r}$ generates the torsion subgroup; if $r \equiv 0 \,\,\mathrm{mod}\,\, 4$ then $t$ generates the torsion subgroup.

If $r$ is odd, the map (\ref{eq:TorsMap}) sends the generator $t^{a/r}$ to $2\chi(\Sigma_g) \cdot \tfrac{(r-3a)}{\gcd(r,3)}$, which is zero modulo $r$ as $r \divides \chi(\Sigma_g)$. If $r$ is even, the result follows by applying the formul{\ae} above to $t^{0/r}$ if $r \equiv 2 \,\,\mathrm{mod}\,\, 4$ or $t$ if $r \equiv 0 \,\,\mathrm{mod}\,\, 4$.
\end{proof}

From these formul{\ae} one can compute the kernel of the homomorphism (\ref{eq:TorsMap}), as it is cyclic and the above proposition determines its order. This group is then Pontrjagin dual to $H_1(\widetilde{\gM}_g^{1/r}[\epsilon];\bZ)$. As one can readily see from the statement of the above proposition, there will not be an especially pleasant formula for the order of this group, as for any fixed $r$ it varies with $\epsilon$ and $g$.


We can also use the formul{\ae} above to compute the second integral cohomology of $\widetilde{\gM}_g^{1/r}[\epsilon]$. Of course as an abstract group it is isomorphic to $\bZ \oplus H_1(\widetilde{\gM}_g^{1/r}[\epsilon];\bZ)$ and the calculation of first homology follows from Proposition \ref{prop:EffectOnH1}. However the formul{\ae} above also let us calculate the image of the injective map
$$H^2(\widetilde{\gM}_g^{1/r}[\epsilon];\bZ) \lra H^2(\gM_g^{1/r}[\epsilon];\bZ)$$
because the Leray--Serre spectral sequence for the fibration (\ref{eq:ExtensionFibration}) gives an exact sequence
$$0 \lra H^2(\widetilde{\gM}_g^{1/r}[\epsilon];\bZ) \lra H^2(\gM_g^{1/r}[\epsilon];\bZ) \lra \bZ/r$$
and the last map is determined by the formul{\ae} above. As these are fairly complicated, it is difficult to give a general expression for the image of this map. However, we have calculated the following three examples. We refer to Examples \ref{ex:rEq2}--\ref{ex:rEq4} for presentations of the second cohomology of $\gM_g^{1/r}[\epsilon]$ in these cases.

\begin{example}\label{ex:ThetaChar2Spin}
The group $H^2(\widetilde{\gM}_g^{1/2}[\epsilon];\bZ)$ is isomorphic to $\bZ \oplus \bZ/4$ for $g \geq 9$, but the map to $H^2(\gM_g^{1/2}[\epsilon];\bZ)$ is an isomorphism if $\epsilon=0$ and an injection onto the index 2 subgroup $\langle \lambda, 2\mu \,\, \vert \,\, 4(\lambda + 4\mu) \rangle$ if $\epsilon=1$.
\end{example}

\begin{example}
The group $H^2(\widetilde{\gM}_g^{1/3};\bZ)$ is isomorphic to $\bZ \oplus \bZ/3$ for $g \geq 9$, and the map to $H^2(\gM_g^{1/3};\bZ)$ is an isomorphism.
\end{example}

\begin{example}
The map from $H^2(\widetilde{\gM}_g^{1/4}[\epsilon];\bZ)$ to $H^2({\gM}_g^{1/4}[\epsilon];\bZ)$ is an isomorphism if $\epsilon=0$, and an injection on to the index 2 subgroup $\langle 2\mu, \lambda^{1/4} \,\, \vert \,\, 4(2\mu - 4\lambda^{1/4}) \rangle$ if $\epsilon=1$.
\end{example}

\section{Relating Picard groups to cohomology}

In this section we will define the necessary terms and prove Theorem \ref{thm:PicardMain} from the introduction.

\subsection{Topological Picard groups}
Recall from the introduction that the \emph{topological Picard group} $\mathrm{Pic}_{\top}(X \hcoker G)$ of a global quotient orbifold $X \hcoker G$ is the set of isomorphism classes of $G$-equivariant complex line bundles on $X$, which is an abelian group under tensor product of line bundles. The first Chern class provides a map
$$c_1 : \mathrm{Pic}_{\top}(X \hcoker G) \lra H^2(X \hcoker G;\bZ)$$
which is an isomorphism \cite[Lemma 5.1]{ERW10}. Thus the maps
\begin{equation*}
\mathrm{Pic}_{\top}(\widetilde{\gM}_g^{1/r}[\epsilon]) \overset{c_1}\lra H^2( \widetilde{\gM}_g^{1/r}[\epsilon];\bZ) \quad\quad \mathrm{Pic}_{\top}( \gM_g^{1/r}[\epsilon]) \overset{c_1}\lra H^2(\gM_g^{1/r}[\epsilon];\bZ)
\end{equation*}
are both isomorphisms. 

\subsection{Holomorphic Picard groups}

Recall from the introduction that if a group $\Gamma$ acts by biholomorphisms on a complex manifold $X$, the orbifold quotient $\gX := X\hcoker \Gamma$ has a complex structure and we may define the \emph{holomorphic Picard group} $\mathrm{Pic}_{\hol}(\gX)$ to be the set of isomorphism classes of $G$-equivariant holomorphic line bundles on $X$, which is an abelian group under tensor product of line bundles. Equivalently, if $\mathcal{O}_\gX$ is the sheaf of holomorphic functions on the orbifold $\gX$, and $\mathcal{O}^\times_\gX$ the subsheaf of nowhere zero functions, it is the sheaf cohomology group $H^1(\gX;\mathcal{O}^\times_\gX)$. The first Chern class provides a map
\begin{equation}\label{eq:HolChern}
c_1 : \mathrm{Pic}_{\hol}(\gX) \lra H^2(\gX;\bZ),
\end{equation}
which coincides with the connecting homomorphism for the exponential sequence of sheaves $0 \to \bZ \to \mathcal{O}_\gX \to \mathcal{O}^\times_\gX \to 0$ on $\gX$. With the identification of the topological Picard group in the previous section, the long exact sequence for the exponential sequence identifies the subgroup $\Pic_{\hol}^0(\gX) < \Pic_{\hol}(\gX)$ of topologically trivial bundles as the quotient
$$\Pic_{\hol}^0(\gX) = H^1(\gX;\mathcal{O}_\gX) / H^1(\gX;\bZ)$$
and so as a connected abelian topological group. The \emph{analytic Neron--Severi group} of $\gX$ is defined to be
$$\mathcal{NS}(\gX) := \Pic_{\hol}(\gX) / \Pic_{\hol}^0(\gX),$$
which is also the subgroup of $\Pic_{\top}(\gX)$ of those complex line bundles which admit a holomorphic structure. The first Chern class descends to an injective homomorphism $c_1 : \mathcal{NS}(\gX) \to H^2(\gX;\bZ)$.

\begin{prop}
The map $c_1 : \Pic_{\hol}(\gM_g^{1/r}[\epsilon]) \to H^2(\gM_g^{1/r}[\epsilon];\bZ)$ is surjective, so $\mathcal{NS}(\gM_g^{1/r}[\epsilon]) \to H^2(\gM_g^{1/r}[\epsilon];\bZ)$ is an isomorphism.
\end{prop}
\begin{proof}
By the exponential sequence, the cokernel of $c_1$ is a subgroup of the complex vector space $H^2\left(\gM_g^{1/r}[\epsilon] ; \mathcal{O}_{\gM_g^{1/r}[\epsilon]} \right)$ and hence is torsion-free. Thus it is enough to see that $c_1$ is rationally surjective, but rationally the cohomology has rank 1 and is generated by the Hodge class $\lambda$, which is pulled back from a holomorphic (in fact algebraic) line bundle on $\gM_g$.
\end{proof}

There are homomorphisms $\Pic_{\top / \hol}(\gM_g^{1/r}[\epsilon]) \to \bZ/r$ given by evaluating a line bundle, holomorphic or topological, on the sub-orbifold $* \hcoker (\bZ/r)$. This gives a commutative diagram
\begin{diagram}
 & & \Pic_{\hol}(\widetilde{\gM}_g^{1/r}[\epsilon]) & \rTo & \Pic_{\hol}(\gM_g^{1/r}[\epsilon]) & \rTo & \bZ/r\\
& & \dTo^{c_1} & & \dTo^{c_1} & & \dEq\\
0 & \rTo & H^2(\widetilde{\gM}_g^{1/r}[\epsilon];\bZ) & \rTo & H^2(\gM_g^{1/r}[\epsilon];\bZ) & \rTo & \bZ/r &\\
\end{diagram}
with exact rows, as a holomorphic line bundle on $\gM_g^{1/r}[\epsilon]$ descends to $\widetilde{\gM}_g^{1/r}[\epsilon]$ precisely when it is trivial on $* \hcoker (\bZ/r)$. This implies

\begin{prop}
The map $c_1 : \Pic_{\hol}(\widetilde{\gM}_g^{1/r}[\epsilon]) \to H^2(\widetilde{\gM}_g^{1/r}[\epsilon];\bZ)$ is surjective, so $\mathcal{NS}(\widetilde{\gM}_g^{1/r}[\epsilon]) \to H^2(\widetilde{\gM}_g^{1/r}[\epsilon];\bZ)$ is an isomorphism.
\end{prop}

\subsection{Algebraic Picard groups}\label{sec:AlgPicGp}

Let us say a complex orbifold $\gX$ is \emph{quasi-projective} if it has a finite cover by a smooth quasi-projective variety. The generalised Riemann existence theorem implies that a finite cover of a quasi-projective variety is again a quasi-projective variety, so this notion does not depend on a choice of quasi-projective cover.

\begin{prop}
The complex orbifold $\widetilde{\gM}_g^{1/r}[\epsilon]$ is quasi-projective.
\end{prop}
\begin{proof}
The intersection of $G_g^{1/r}(\zeta)$ with the third level subgroup $\Gamma_g[3] \lhd \Gamma_g$ acts freely on Teichm\"{u}ller space, with quotient a finite unramified cover of $\mathcal{M}_g[3]$, and hence a smooth quasi-projective variety.
\end{proof}

For quasi-projective orbifolds, there is yet another notion of Picard group available. The \emph{algebraic Picard group} $\Pic_{\alg}(\gX)$ is the set of isomorphism classes of holomorphic line bundles on $\gX$ which are algebraic on the quasi-projective cover of $\gX$. Alternatively, if $\overline{\mathcal{O}}^\times_\gX$ is the sheaf of nowhere zero holomorphic functions on $\gX$ which are algebraic on the quasi-projective cover, we may define
$$\Pic_{\alg}(\gX) := H^1(\gX; \overline{\mathcal{O}}_\gX^\times).$$

As in the holomorphic case, we define $\Pic^0_{\alg}(\gX)$ to be the subgroup of topologically trivial algebraic line bundles. The additional algebraic structure implies that $\Pic^0_{\alg}(\gX) =0$ if $H^1(\gX;\bZ)=0$, c.f.\ \cite[Theorem 14.3]{HainTransc}. In particular, this is the case for $\widetilde{\gM}_g^{1/r}[\epsilon]$, so

\begin{prop}\label{prop:InjAlgPicTilde}
The composition of maps
$$\Pic_{\alg}(\widetilde{\gM}_g^{1/r}[\epsilon]) \lra \Pic_{\hol}(\widetilde{\gM}_g^{1/r}[\epsilon]) \lra \mathcal{NS}(\widetilde{\gM}_g^{1/r}[\epsilon]) \overset{c_1}\lra H^2(\widetilde{\gM}_g^{1/r}[\epsilon];\bZ)$$
is injective. Moreover, the second map is surjective and the last is an isomorphism.
\end{prop}

If there is a $G$-gerbe $\gY \overset{\pi}\to \gX$ where $G$ is a compact abelian group and $\gX$ is a quasi-projective orbifold, we say that $\gY$ is a \emph{small extension of a quasi-projective orbifold.} In this case the map $\pi$ has local sections, and so there is a natural isomorphism $\mathcal{O}_\gY \cong \pi^{-1}\mathcal{O}_\gX$ of sheaves of holomorphic functions. Hence we may define the sub-sheaf of algebraic functions on $\gY$ by $\overline{\mathcal{O}}_\gY := \pi^{-1}\overline{\mathcal{O}}_\gX$, and the algebraic Picard group by $\Pic_{\alg}(\gY) := H^1(\gY; \overline{\mathcal{O}}_\gY^\times)$. In particular applying this to the $\bZ/r$-gerbe $\widetilde{\gM}_g^{1/r}[\epsilon] \to \gM_g^{1/r}[\epsilon]$ we may define $\Pic_{\alg}(\gM_g^{1/r}[\epsilon])$.

\begin{prop}
The composition of maps
$$\Pic_{\alg}({\gM}_g^{1/r}[\epsilon]) \lra \Pic_{\hol}({\gM}_g^{1/r}[\epsilon]) \lra \mathcal{NS}({\gM}_g^{1/r}[\epsilon]) \overset{c_1}\lra H^2({\gM}_g^{1/r}[\epsilon];\bZ)$$
is injective. Moreover, the second map is surjective and the last is an isomorphism.
\end{prop}
\begin{proof}
For any such gerbe, applying the Leray spectral sequence to the map $\pi$ and the sheaf $\overline{\mathcal{O}}_\gY^\times$, we obtain the exact sequence
$$0 \lra \Pic_{\alg}(\gX) \lra \Pic_{\alg}(\gY) \lra H^2(G;\bZ).$$
For $\widetilde{\gM}_g^{1/r}[\epsilon] \to \gM_g^{1/r}[\epsilon]$ we obtain a commutative diagram with exact rows
\begin{diagram}[LaTeXeqno]\label{dig:Alg}
0 &\rTo& \Pic_{\alg}(\widetilde{\gM}_g^{1/r}[\epsilon]) &\rTo& \Pic_{\alg}(\gM_g^{1/r}[\epsilon]) &\rTo& \bZ/r\\
 & & \dTo & & \dTo & & \dEq\\
 0 &\rTo& H^2(\widetilde{\gM}_g^{1/r}[\epsilon];\bZ) &\rTo& H^2(\gM_g^{1/r}[\epsilon];\bZ) &\rTo& \bZ/r
\end{diagram}
with the left vertical map injective by Proposition \ref{prop:InjAlgPicTilde}. Hence the middle vertical map is also injective.
\end{proof}

We can also extract information on the image of the injective homomorphism $\Pic_{\alg}({\gM}_g^{1/r}[\epsilon]) \lra H^2({\gM}_g^{1/r}[\epsilon];\bZ)$. First we note that all the classes $\lambda^{-a/r}$ are in the image of this homomorphism, as the construction $\lambda^{-a/r} := c_1(\pi_!^K(L^{\otimes a}))$ can be performed algebraically as $\det(\pi_!(L^{\otimes a}))$ which gives a class in $\Pic_{\alg}({\gM}_g^{1/r}[\epsilon])$ with Chern class $\lambda^{-a/r}$. For $r$ odd these classes generate $H^2({\gM}_g^{1/r}[\epsilon];\bZ)$ and so the homomorphism is surjective. For $r$ even, the classes $\lambda^{a/r}$ only generate an index 2 subgroup of $H^2({\gM}_g^{1/r}[\epsilon];\bZ)$, but together with $\mu$ they generate the entire group.

\begin{lem}
The map $c_1 : \Pic_{\alg}({\gM}_g^{1/2}[\epsilon]) \to H^2({\gM}_g^{1/2}[\epsilon];\bZ)$ is surjective.
\end{lem}
\begin{proof}
The coarse moduli space $\widetilde{\mathrm{M}}_g^{1/2}$ is in natural bijection with the set of pairs $(\Sigma, L)$ of a Riemann surface and a complex line bundle $L$ such that $L^{\otimes 2} \cong T\Sigma$. As such, $L$ admits a canonical structure of a holomorphic line bundle on $\Sigma$, and the holomorphic structure is independent of the choice of isomorphism $L^{\otimes 2} \cong T\Sigma$. At this point it is convenient to work with square roots of the cotangent bundle, so let us denote by $L^*$ the dual bundle to $L$.

There is a function $\rho : \widetilde{\mathrm{M}}_g^{1/2} \to \bN$ given by $\rho(\Sigma, L) = \mathrm{dim}H^0(\Sigma;L^*)$, where  we identify $L^*$ with its sheaf of algebraic sections. The parity of $\mathrm{dim}H^0(\Sigma;L^*)$ is locally constant on $\widetilde{\mathrm{M}}_g^{1/2}$, and agrees with the Arf invariant, but the function itself is not locally constant. The subset
$$\Theta_{\mathrm{null}} := \{(\Sigma, L) \in \widetilde{\mathrm{M}}_g^{1/2}[0] \,\, \vert \,\, \rho(\Sigma, L) > 0\}$$
is classically known to be a divisor, and the subset
$$T := \{(\Sigma, L) \in \widetilde{\mathrm{M}}_g^{1/2}[1] \,\, \vert \,\, \rho(\Sigma, L) > 1\}$$
is known to have codimension 3 as long as $g \geq 5$ \cite[Theorem 2.17]{Teixidor}.

We must now proceed by cases depending on the parity. For $\epsilon=0$, by \cite[Theorem 0.2]{FarkasEven}, the class of $\Theta_{\mathrm{null}}$ is rationally equivalent to $-\tfrac{1}{4}\lambda$ on $\widetilde{\mathrm{M}}_g^{1/2}[0]$, and it remains so when pulled back to $\widetilde{\gM}_g^{1/2}[0]$ or ${\gM}_g^{1/2}[0]$. Recall from Example \ref{ex:rEq2} that $H^2({\gM}_g^{1/2}[0];\bZ)$ is isomorphic to $\langle \lambda, \mu \,\, \vert \,\, 4(\lambda+4\mu) \rangle$. Then $4\cdot c_1(\Theta_{\mathrm{null}}) = -\lambda = 4\mu$ in the torsion-free quotient of $H^2({\gM}_g^{1/2}[0];\bZ)$, and hence
$$c_1(\Theta_{\mathrm{null}}) = \mu + A(\lambda+4\mu)$$
for some $A$. However, $\lambda+4\mu = \lambda + 2(\lambda^{-1/2} + 6\lambda^{1/2})$ is already in the image of $c_1$, and so $\mu$ is also.

For $\epsilon=1$, note that on the complement $\widetilde{\gM}_g^{1/2}[1] \setminus T$ the sheaf $(\Sigma, L) \mapsto H^0(\Sigma;L^*)$ (equivalently, the direct image $\pi_*(L^*)$) is a vector bundle of rank 1, and hence represents an element of $\Pic_{\alg}(\widetilde{\gM}_g^{1/2}[1] \setminus T)$. As $T$ as codimension 3, this extends uniquely to an element of $\Pic_{\alg}(\widetilde{\gM}_g^{1/2}[1])$. To compute the first Chern class of this line bundle, note that
$$\pi_!(L^*) = \pi_*(L^*) - R^1\pi_*(L^*)$$
as virtual sheaves on $\widetilde{\gM}_g^{1/2}[1]$, but the isomorphism $L^* \cong T^*\Sigma \otimes L$ and Serre duality shows that $R^1\pi_*(L^*) \cong (\pi_*(L^*))^*$ and so $\pi_!(L^*) = \pi_*(L^*) - \pi_*(L^*)^*$. Taking first Chern classes (over $\widetilde{\gM}_g^{1/2}[1] \setminus T$, where $\pi_*(L^*)$ is a bundle) gives
$$2\cdot c_1(\pi_*(L^*)) = \lambda^{-1/2},$$
which is $2\mu-12\lambda^{1/2}$ rationally. Thus $c_1(\pi_*(L^*)) = \mu-6\lambda^{1/2}$ modulo torsion, and as we saw above the generator $\lambda + 4\mu$ of the torsion subgroup is already known to be in the image, as is $\lambda^{1/2}$. Hence $\mu$ is too.
\end{proof}

This lemma implies that $c_1 : \Pic_{\alg}({\gM}_g^{1/r}) \to H^2({\gM}_g^{1/r};\bZ)$ is surjective for all even $r$, and hence an isomorphism, as the missing class $\mu$ is pulled back from ${\gM}_g^{1/2}$. The commutative diagram (\ref{dig:Alg}) implies that the same is true for $\widetilde{\gM}_g^{1/r}$.

\appendix

\section{A calculation in homotopy theory}\label{sec:Appendix:Calculation}
In this appendix we compute the first Chern class of the K-theory class
$$\Omega^\infty \beta^{-1} : \Omega^{\infty} \gS^{-2} \simeq \Omega^2 Q(S^0) \lra \Omega^\infty \gK \simeq \bZ \times BU,$$
which completes the proof of Proposition \ref{prop:ClassesOnQS0}. As $\Omega^2 c_2 = c_1$ under the Bott equivalence, this the same as computing the second Chern class of the unit $Q(S^0) \lra \bZ \times BU$, and then computing its double loop class.

In order to do this, we make use of the fact that the fibrations of infinite loop spaces
$$\widetilde{Q_0S^0} \lra Q_0S^0 \lra K(\bZ/2,1)$$
and
$$\widetilde{\Omega_0 QS^0} \lra \Omega ( \widetilde{Q_0S^0}) \lra K(\bZ/2,1)$$
are split as fibrations of spaces: a proof of this fact appears in \cite[Lemma 7.1]{galatius-2004}. 

We will first treat the 2-torsion. We may write down a canonical basis for the $\bF_2$-homology of $Q_0S^0$, in terms of Dyer--Lashof operations, and we include this information up to degree four in Figure \ref{fig:QS0}.
\begin{figure}[h]
\centering
\includegraphics[bb=0 0 327 110]{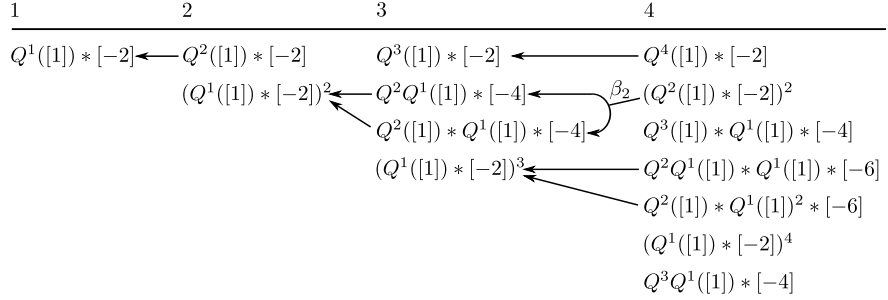}
\caption{$\bF_2$-homology of $Q_0S^0$ up to degree four. All the arrows denote primary Bockstein operations, except the indicated secondary operation. The primary operations follow from the Nishida relation $\beta Q^b = (b-1)Q^{b-1}$, and the secondary operation follows from the formula of \cite[Lemma 4.11]{CLM}.}
\label{fig:QS0}
\end{figure}
From this description we are able to calculate the 2-local homology of $Q_0S^0$ up to degree three, which we give in the following table.
\begin{center}
\begin{tabular}{lccccc}
\toprule
$i$ & $0$ & $1$ & $2$ & $3$ & $4$ \\ \toprule
$H_*(Q_0S^0;\bF_2)$ & $\bF_2$ & $\bF_2$ & $\bF_2^2$ & $\bF_2^4$ & $\bF_2^7$  \\ 
$H_*(Q_0S^0;\bZ_{(2)})$ & $\bZ_{(2)}$ & $\bZ/2$ & $\bZ/2$ & $(\bZ/2)^2 \oplus \bZ/4$ & $-$ \\ \bottomrule
\end{tabular}
\end{center}
Using the splitting of spaces $Q_0S^0 \simeq K(\bZ/2, 1) \times \widetilde{Q_0S^0}$, the known homology of $K(\bZ/2, 1)$, the K\"{u}nneth theorem and the universal coefficient theorem, we are able to compute the homology of $\widetilde{Q_0S^0}$ up to degree three, and so its cohomology up to degree four, which we give in the following table.
\begin{center}
\begin{tabular}{lccccc}
\toprule
$H_*(\widetilde{Q_0S^0};\bF_2)$ & $\bF_2$ & $0$ & $\bF_2$ & $\bF_2^2$ & $\bF_2^3$\\ 
$H_*(\widetilde{Q_0S^0};\bZ_{(2)})$ & $\bZ_{(2)}$ & $0$ & $\bZ/2$ & $\bZ/4$ & $-$\\ 
$H^*(\widetilde{Q_0S^0};\bZ_{(2)})$ & $\bZ_{(2)}$ & $0$ & $0$ & $\bZ/2$ & $\bZ/4$\\ \bottomrule
\end{tabular}
\end{center}
The universal cover $\widetilde{\Omega_0 QS^0}$ is simply-connected and has second homotopy group equal to $\pi_3(QS^0) = \bZ/24$. Thus we can compute its 2-local homology up to degree two. Using the splitting $\Omega \widetilde{Q_0S^0} \simeq K(\bZ/2, 1) \times \widetilde{\Omega_0QS^0}$ we are able to compute the 2-local homology of $\Omega \widetilde{Q_0S^0}$ up to degree two, and hence its cohomology up to degree three.
\begin{center}
\begin{tabular}{lccccc}
\toprule
$H_*(\widetilde{\Omega_0 QS^0};\bZ_{(2)})$ & $\bZ_{(2)}$ & $0$ & $\bZ/8$ & $-$ & $-$\\
$H^*(\widetilde{\Omega_0 QS^0};\bZ_{(2)})$ & $\bZ_{(2)}$ & $0$ & $0$ & $\bZ/8$ & $-$\\
$H_*(\Omega\widetilde{Q_0S^0};\bZ_{(2)})$ & $\bZ_{(2)}$ & $\bZ/2$ & $\bZ/8$ & $-$ & $-$\\
$H^*(\Omega\widetilde{Q_0S^0};\bZ_{(2)})$ & $\bZ_{(2)}$ & $0$ & $\bZ/2$ & $\bZ/8$ & $-$\\
\bottomrule
\end{tabular}
\end{center}

Using the known structure \cite{Priddy} of the homology of $\bZ \times BU$ as a module over the Dyer--Lashof algebra, it is not difficult to verify that $c_2$ reduces modulo 2 to the dual of the homology class $(Q^2([1])*[-2])^2$, and so generates the group $H^4(\widetilde{Q_0S^0};\bZ_{(2)}) = \bZ/4$.

We now consider the Serre spectral sequence for the loop-space fibration
$$\Omega \widetilde{Q_0S^0} \lra * \lra \widetilde{Q_0S^0}.$$
We see that the only differentials out of $E^2_{0, 3} = \bZ/8$ are to $E^2_{2,2} = \bZ/2$ and $E^2_{4,0} = \bZ/4$. By cardinality both of these must be surjective, and in particular $2 \in \bZ/8 = E^2_{0,3}$ transgresses to a generator of $E^2_{4,0} = \bZ/4$. Thus the loop of a generator of $H^4(\widetilde{Q_0S^0};\bZ_{(2)})$ gives twice a generator in $H^3(\Omega \widetilde{Q_0S^0};\bZ_{(2)})$.

By our calculations $H^3(\Omega \widetilde{Q_0S^0};\bZ_{2)}) \to H^3(\widetilde{\Omega_0QS^0};\bZ_{(2)})$ is an isomorphism. Consider the Serre spectral sequence for the loop-space fibration
$$\Omega^2_0 QS^0 \simeq \Omega (\widetilde{\Omega_0 QS^0}) \lra * \lra \widetilde{\Omega_0 QS^0},$$
and our calculations show that the transgression
$$\tau: H^2(\Omega^2_0 QS^0;\bZ_{(2)}) \lra H^3(\widetilde{\Omega_0 QS^0};\bZ_{(2)})$$
is an isomorphism. Hence it follows that $\Omega^2 c_2$ is twice a generator in $H^2(\Omega^2_0 QS^0;\bZ_{(2)})$.

\vspace{2ex}

The 3-local calculation is of course much easier. By \cite[Theorem 22]{Kochman}, we have that $Q^1([1])*[-3] \in H_4(\{0\} \times BU ;\bF_3)$ is the linear dual of the second Chern class, and hence the unit $Q_0S^0 \to BU$ pulls back the second Chern class to the linear dual of $Q^1([1])*[-3] \in H_4(Q_0S^0 ;\bF_3)$. As the splittings of $Q_0S^0$ and $\Omega(\widetilde{Q_0S^0})$ have factors which are trivial 3-locally, the situation is drastically simpler. It is easy to see that
\begin{center}
\begin{tabular}{lccccc}
\toprule
$i$ & $0$ & $1$ & $2$ & $3$ & $4$ \\ \toprule
$H_*(\widetilde{Q_0S^0};\bF_3)$ & $\bF_3$ & $0$ & $0$ & $\bF_3$ & $\bF_3$  \\ 
$H_*(\widetilde{Q_0S^0};\bZ_{(3)})$ & $\bZ_{(3)}$ & $0$ & $0$ & $\bZ/3$ & $0$ \\
$H^*(\widetilde{Q_0S^0};\bZ_{(3)})$ & $\bZ_{(3)}$ & $0$ & $0$ & $0$ & $\bZ/3$ \\ 
\bottomrule
\end{tabular}
\end{center}
and that the second Chern class generates $H^4(\widetilde{Q_0S^0};\bZ_{(3)})$. Hence it's second loop class generates $H^2(\Omega^2_0QS^0;\bZ_{(3)}) \cong \bZ/3$ as well.

\bibliographystyle{plain}
\bibliography{../MainBib}

\def\cprime{$'$}
\begin{thebibliography}{10}

\bibitem{BG}
J.~C. Becker and D.~H. Gottlieb.
\newblock The transfer map and fiber bundles.
\newblock {\em Topology}, 14:1--12, 1975.

\bibitem{CLM}
Frederick~R. Cohen, Thomas~J. Lada, and J.~Peter May.
\newblock {\em The homology of iterated loop spaces}.
\newblock Springer-Verlag, Berlin, 1976.
\newblock Lecture Notes in Mathematics, Vol. 533.

\bibitem{Cornalba}
Maurizio Cornalba.
\newblock Moduli of curves and theta-characteristics.
\newblock In {\em Lectures on {R}iemann surfaces ({T}rieste, 1987)}, pages
  560--589. World Sci. Publ., Teaneck, NJ, 1989.

\bibitem{Cornalba2}
Maurizio Cornalba.
\newblock A remark on the {P}icard group of spin moduli space.
\newblock {\em Atti Accad. Naz. Lincei Cl. Sci. Fis. Mat. Natur. Rend. Lincei
  (9) Mat. Appl.}, 2(3):211--217, 1991.

\bibitem{EE}
C.J. Earle and J.~Eells.
\newblock A fibre bundle description of {T}eichmueller theory.
\newblock {\em J. Differential Geometry}, 3:19--43, 1969.

\bibitem{ERW10}
Johannes Ebert and Oscar Randal-Williams.
\newblock Stable cohomology of the extended mapping class group and the
  universal {P}icard variety.
\newblock 2010.

\bibitem{FarkasEven}
Gavril Farkas.
\newblock The birational type of the moduli space of even spin curves.
\newblock {\em Adv. Math.}, 223(2):433--443, 2010.

\bibitem{galatius-2004}
S{\o}ren Galatius.
\newblock Mod {$p$} homology of the stable mapping class group.
\newblock {\em Topology}, 43(5):1105--1132, 2004.

\bibitem{GMTW}
S{\o}ren Galatius, Ib~Madsen, Ulrike Tillmann, and Michael Weiss.
\newblock The homotopy type of the cobordism category.
\newblock {\em Acta Math.}, 202(2):195--239, 2009.

\bibitem{GR-W}
S{\o}ren Galatius and Oscar Randal-Williams.
\newblock Monoids of moduli spaces of manifolds.
\newblock {\em Geom. Topol.}, 14(3):1243--1302, 2010.

\bibitem{HainTransc}
Richard Hain.
\newblock Moduli of {R}iemann surfaces, transcendental aspects.
\newblock In {\em School on {A}lgebraic {G}eometry ({T}rieste, 1999)}, volume~1
  of {\em ICTP Lect. Notes}, pages 293--353. Abdus Salam Int. Cent. Theoret.
  Phys., Trieste, 2000.

\bibitem{JarvisGeom}
Tyler~J. Jarvis.
\newblock Geometry of the moduli of higher spin curves.
\newblock {\em Internat. J. Math.}, 11(5):637--663, 2000.

\bibitem{Jarvis}
Tyler~J. Jarvis.
\newblock The {P}icard group of the moduli of higher spin curves.
\newblock {\em New York J. Math.}, 7:23--47 (electronic), 2001.

\bibitem{Kochman}
Stanley~O. Kochman.
\newblock Homology of the classical groups over the {D}yer--{L}ashof algebra.
\newblock {\em Trans. Amer. Math. Soc.}, 185:83--136, 1973.

\bibitem{MW}
Ib~Madsen and Michael Weiss.
\newblock The stable moduli space of {R}iemann surfaces: {M}umford's
  conjecture.
\newblock {\em Ann. of Math. (2)}, 165(3):843--941, 2007.

\bibitem{Priddy}
Stewart Priddy.
\newblock Dyer--{L}ashof operations for the classifying spaces of certain
  matrix groups.
\newblock {\em Quart. J. Math. Oxford Ser. (2)}, 26(102):179--193, 1975.

\bibitem{R-WResolution}
Oscar Randal-Williams.
\newblock Resolutions of moduli spaces.
\newblock arXiv:0909.4278, 2009.

\bibitem{RWFramedPinMCG}
Oscar Randal-Williams.
\newblock Homology of the moduli spaces and mapping class groups of framed,
  $r$-{S}pin and {P}in surfaces.
\newblock 2010.
\newblock arXiv:1001.5366.

\bibitem{Teixidor}
Montserrat Teixidor~i Bigas.
\newblock Half-canonical series on algebraic curves.
\newblock {\em Trans. Amer. Math. Soc.}, 302(1):99--115, 1987.

\end{thebibliography}

\end{document}